\documentclass[a4paper,11pt]{article}

\usepackage{authblk}
\usepackage{a4wide}
\usepackage[utf8]{inputenc}
\usepackage{amsmath}
\usepackage{amsfonts}
\usepackage{amssymb}
\usepackage{mathrsfs}
\usepackage{amsthm}
\usepackage{color}
\usepackage{stmaryrd}
\usepackage[hidelinks]{hyperref}

\usepackage{latexsym,enumerate}
\usepackage{xcolor}
\usepackage{graphicx}
\usepackage{epstopdf}

\usepackage[T1]{fontenc}
\usepackage[utf8]{inputenc}
\usepackage[english]{babel}

\newcommand{\VV}[2]{\left(\!\begin{array}{c}#1\\#2\end{array}\!\right) }
\newcommand{\dx}[0]{\,\mathrm{d}}

\newcommand{\fbsde}[0]{}
\newcommand{\esssup}[0]{\mathrm{ess}\,\mathrm{sup}}
\definecolor{grey2}{gray}{0.72}
\definecolor{grey1}{gray}{0.87}
\newcommand{\mybox}[1]{\textrm{\colorbox{grey1}{$#1$}}}
\newcommand{\Mybox}[1]{\textrm{\colorbox{grey2}{$#1$}}}
\newcommand{\myqed}[0]{\checkmark}

\def \1{\mathbf{1}}

\newcommand{\be}{\begin{eqnarray*}}
\newcommand{\ee}{\end{eqnarray*}}
\newcommand{\ben}{\begin{eqnarray}}
\newcommand{\een}{\end{eqnarray}}
\newcommand{\bi}{\begin{itemize}}
\newcommand{\ei}{\end{itemize}}

\definecolor{color2}{gray}{0.7}

\newtheorem{thm}{Theorem}[section]

\newtheorem{lemma}[thm]{Lemma}

\theoremstyle{definition}

\newtheorem{defi}[thm]{Definition}

\newtheorem{remark}[thm]{Remark}

\date{April 12, 2017}
\title{Utility maximization via decoupling fields}
\author[1]{Alexander Fromm\thanks{alexander.fromm@uni-jena.de}}
\author[2]{Peter Imkeller\thanks{imkeller@math.hu-berlin.de}}
\affil[1]{\small Institute for Mathematics, University of Jena, Ernst-Abbe-Platz 2, 07743 Jena, Germany}
\affil[2]{\small Institute for Mathematics, Humboldt University of Berlin, Unter den Linden 6, D-10099 Berlin}

\begin{document}

\maketitle

\begin{abstract}
We consider the utility maximization problem for a general class of utility functions defined on the real line. We rely on existing results which reduce the problem to a coupled forward-backward stochastic differential equation (FBSDE) and concentrate on showing existence and uniqueness of solution processes to this FBSDE. We use the method of decoupling fields for strongly coupled, multi-dimensional and possibly non-Lipschitz systems as the central technique in conducting the proofs.
\end{abstract}

\vspace{0.5cm}


\section*{Introduction}

The central problem to which we apply techniques of forward-backward stochastic differential equations (FBSDE) in this paper originally comes from \emph{securitization}, i.e. insuring market exogenous risk by investing on a capital market. Typically, a small agent whose preferences are described by a utility function $U$ wants to insure
a random liability $H$ arising from his usual business. He, therefore, has two sources of income: his random liability, and
the wealth obtained from trading on the capital market up to a terminal
time with appropriate investment strategies. The \textit{stochastic control problem}
he faces results in the maximization of his terminal utility obtained from
both sources of income with respect to all admissible strategies. More formally, given his initial wealth $x>0$, he wants to attain
the value
\begin{equation}
\label{eq:introopti}
V(0,x):=\sup_{\pi \in \mathcal{A}} \mathbb{E}[U(X_T^\pi+H)],
\end{equation}
where $\mathcal{A}$ is the set of {\sl admissible trading
strategies}, $T < \infty$ the trading horizon, $X_T^{\pi}$ the agent's terminal
wealth related to an investment strategy $\pi \in \mathcal{A}$.
This wealth is obtained from investing in a financial market composed of a zero interest rate bond, and $d\geq 1$ stocks given by
$$ d{S}_t^i:={S}_t^i dW^i_t + {S}_t^i \theta_t^i dt, \quad i \in \{1,\ldots,d\}, $$
where $W$ is a standard Brownian motion on $\mathbb{R}^d$, and $\theta$ the process describing market risk. As in \cite{HIRZ13}, trading underlies a linear  constraint: assume $d_1+d_2=d$ and that the agent can only invest in the assets ${S}^1,\ldots,{S}^{d_1}$.
Other relevant problems in this context are the characterization of optimal strategies and the \textit{value
function} $V$ which for $0\le t\le T$ is given by
$$ V(t,x):=\sup_{\pi \in \mathcal{A}} \mathbb{E}[U(X_{t,T}^\pi+H)\vert \mathcal{F}_t].$$
Here $X_{t,T}$ is the wealth obtained in the investment period $[t,T]$, and
$(\mathcal{F}_t)_{t\in [0,T]}$ describes the evolution of information.
\\[5pt]
The most common technique employed to obtain optimal strategies
$\pi^\ast$ is related to \textit{convex analysis and duality} (see Bismut \cite{Bismut}, Pliska \cite{Pliska1986}, Karatzas and co-workers (Karatzas et al.
\cite{Karatzas}, \cite{KaratzasXu}, \cite{CvitanicKaratzas}),
Kramkov
and Schachermayer \cite{KramkovSchachermayer}).
A direct stochastic approach to characterize optimal trading
strategies in the case of non-linear, even non-convex trading constraints is provided by an interpretation of the martingale
optimality principle by (forward) backward stochastic
differential equations (FBSDE) (see El Karoui et al. \cite{ElKarouiRouge}, Sekine
\cite{Sekine}, and Hu et al. \cite{Hu2005}). If
the utility function is exponential, or power or logarithmic (and $H = 0$), it has been shown by Hu et al. \cite{Hu2005} that
the control problem \eqref{eq:introopti} can essentially be reduced to
solving a BSDE of the form
\begin{equation}
\label{eq:introBSDE}
Y_t =H-\int_t^T Z_s dW_s-\int_t^T f(s,Z_s) ds, \quad t\in [0,T],
\end{equation}
where the {\sl driver} $f(t,z)$ is of quadratic
growth in the $z$-variable. \\
While for these classical utility functions forward and backward components of the investment dynamics decouple, in \cite{HIRZ13} the problem \eqref{eq:introopti} has been tackled for a larger class of utility
functions, and shown to lead to a
fully-coupled system of FBSDE, again typically with a driver of quadratic growth in the control variable.
The derivation of this system starts with
a verification type observation. Given an optimal strategy $\pi^\ast$ of the
(forward) portfolio process $X^{\pi^\ast}$, to realize martingale
optimality one postulates that $U'(X^{\pi^\ast}+Y)$ is a martingale, where
$(Y,Z)$ is the associated backward process. As a consequence, $(Y,Z)$ is
given by a certainty equivalent type expression for marginal utility $Y =
(U')^{-1}(\mathbb{E}(U'(X_T^{\pi^\ast}+H)|\mathcal{F}_\cdot) -
X^{\pi^\ast}.$ This allows to compute the driver of the
BSDE related to $(Y,Z)$. It is given in terms of the derivatives of $U$,
involves the optimal forward process $X^{\pi^\ast},$ and provides the
backward part of the FBSDE system. In a second step, one considers possible
solution triples $(X,Y,Z)$ of the FBSDE system obtained in the first step,
not assuming that $X$ corresponds to an optimal portfolio process. One then
uses the variational maximum principle in order to verify that under mild
conditions on $U$ the triple $(X,Y,Z)$ solves the original optimization
problem. This in particular means that $X$ coincides with an optimal
forward portfolio process $X^{\pi^\ast}.$ In summary, under mild regularity
conditions, solutions $(X,Y,Z)$ of the FBSDE system provide
solutions of the original securitization problem.
If $\theta$ is the price of risk process associated to the price dynamics of the market, and $\pi_1$ denotes the projection on the first $d_1$ coordinates in $\mathbb{R}^d$, $\pi_2$ the one on the remaining $d_2$ coordinates, $\pi^\ast$ is given by
\begin{equation}\label{heyne-fbsde:opt:strat}
\pi^\ast=-\pi_1(\theta)\frac{U'(X+Y)}{U''(X+Y)}-\pi_1(Z),
\end{equation}
and the FBSDE by
\begin{multline}\label{utility-FBSDE}
X_t =x-\int_0^t\left(\pi_1(\theta_s)\frac{U'(X_s+Y_s)}{U''(X_s+Y_s)}+\pi_1(Z_s)\right)^\top \dx W_s-\\
-\int_0^t\left(\pi_1(\theta_s)\frac{U'(X_s+Y_s)}{U''(X_s+Y_s)}+\pi_1(Z_s)\right)^\top\pi_1(\theta_s) \dx s, \\
\shoveleft{Y_t =H-\int_t^T Z^\top_s \dx W_s-\int_t^T\left[-\frac{1}{2}|\pi_1(\theta_s)|^2\frac{U^{(3)}(X_s+Y_s)|U'(X_s+Y_s)|^2}{(U''(X_s+Y_s))^3}+\right. }\\
+\left. |\pi_1(\theta_s)|^2\frac{U'(X_s+Y_s)}{U''(X_s+Y_s)}+Z_s\cdot\pi_1(\theta_s)-\frac{1}{2} |\pi_2(Z_s)|^2\frac{U^{(3)}}{U''}(X_s+Y_s)\right]\dx s.
\end{multline}
This of course only translates the original utility optimization problem into another problem the solvability of which is far from obvious and remains largely unanswered in \cite{HIRZ13}.

In this paper we use the technique of decoupling fields to show in a reasonable framework that solution triples of systems as the one above exist and are unique. Let us point out that in this paper we successfully study the real line case, i.e. the class of problems for which the utility function $U$ has $\mathbb{R}$ as its domain, while the equally important half-line case, i.e. the situation where $[0,\infty)$ is the domain, leads to a different FBSDE and is therefore subject to future work.

To sketch the tool of decoupling fields we apply to treat the above forward-backward system, consider a general FBSDE of the form
$$ X_t=x+\int_{0}^t\mu(s,X_s,Y_s,Z_s)\dx s+\int_{0}^t\sigma(s,X_s,Y_s,Z_s)\dx W_s, $$
$$ Y_t=\xi(X_T)-\int_{t}^{T}f(s,X_s,Y_s,Z_s)\dx s-\int_{t}^{T}Z_s\dx W_s.$$
$X$ and $Y$ may be multidimensional, and the particular nature of the underlying problem is encoded in the parameter functions $\mu,\sigma,f$ which may be random, but at least progressively measurable. The terminal condition $\xi$, besides the terminal value of the forward process $X$, may have a further dependence on randomness, and is required to be measurable w.r.t. $\mathcal{F}_T$.
The system is called decoupled if either $\mu,\sigma$ do not depend on $Y,Z$, or if $\xi$,$f$ do not depend on $X$. In these two cases the problem can be treated by solving one of the equations first, and then simply substituting the solution processes obtained into the other one.

Coupled FBSDE have been extensively studied, but
essentially for Lipschitz coefficients (see \cite{mayong}, see also \cite{Delarue}). The so called \emph{Four Step Scheme} (see \cite{mayong}) is based on reducing the problem to a quasi-linear parabolic PDE. This works for parameter functions which are deterministic and sufficiently smooth.
The \emph{Method of Continuation} (e.g. \cite{hupeng}) is purely stochastic, but relies on monotonicity assumptions for the parameter functions that are hard to verify. A more general technique is developed by Zhang et al. with the concept of \emph{Decoupling Fields} in \cite{Ma2011}. The \emph{Contraction Method} proposed by Antonelli \cite{Antonelli1993} (see also \cite{PT}) is extended to construct solutions on large intervals by patching together solutions defined on small intervals. \cite{Ma2011} concentrates on one-dimensional problems, while in Chapter 2 of \cite{Fromm2015} a theory in a multi-dimensional setting has been developed which provides existence as well as uniqueness and regularity of solutions on maximal intervals characterized by properties of the decoupling field. Global Lipschitz continuity is required only in the non-Markovian case, while in the more special Markovian case a form of local Lipschitz continuity suffices.

In this work we apply the general results obtained in Chapter 2 of \cite{Fromm2015} to the particular setting of \eqref{utility-FBSDE} in order to obtain existence and uniqueness of solutions. This requires some conditions concerning the structure of the utility functions considered, expressed by boundedness conditions for quotients of its derivatives up to order 3 together with the condition $\left(\ln(-U'')\right)''\leq 0$ which is motivated in Remark \ref{remarkU}.
The extension of existence and regularity results for decoupling fields to the situation of FBSDE generators with quadratic growth also leads us to consider a (Markovian) scenario in which the market price of risk process together with the terminal condition $\xi$ may depend on randomness, but only through the values of an external, possibly high-dimensional, diffusion. In this scenario, the case in which the driver is only locally Lipschitz is reduced to the Lipschitz case by obtaining effective bounds on the control process through its description by the decoupling field and its derivatives.

The paper is organized as follows. In Section \ref{prelim} we introduce and discuss the assumptions we make on the parameters $U$, $\theta$ and $H$ of the initial problem for the corresponding FBSDE to be solvable using the method of decoupling fields. In the following Section \ref{decf} we briefly sum up the method for a class of problems called MLLC, standing for \emph{Modified Local Lipschitz Conditions}. In the last section we show existence and uniqueness of the FBSDE \eqref{utility-FBSDE} by first showing a more general result of Theorem \ref{sc3} and then applying it to the more specific structure of \eqref{utility-FBSDE} in Theorem \ref{finalresult}.

\section{Preliminaries}\label{prelim}

We require the utility function $U:\mathbb{R}\rightarrow\mathbb{R}$ to satisfy the following condition:

\begin{description}
\item{(C1)} $U$ has the form
$$U(x)=-\int_x^\infty\int_y^\infty \exp(-\kappa(z))\dx z\dx y$$
with some $\kappa:\mathbb{R}\rightarrow\mathbb{R}$ satisfying:
\begin{itemize}
\item $\kappa$ is twice differentiable,
\item $0<\inf_{x\in\mathbb{R}}\kappa'(x)\leq \sup_{x\in\mathbb{R}}\kappa'(x)<\infty$ and
\item $0\leq\inf_{x\in\mathbb{R}}\kappa''(x)\leq \sup_{x\in\mathbb{R}}\kappa''(x)<\infty$.
\end{itemize}
\end{description}

\begin{remark}\label{remarkU}
In the terminology of exponential utility functions, which are a special case of utilities $U$ considered above, $\kappa'(x)$ can be interpreted as the "local risk aversion" which unlike in the case of exponential utility is allowed to change with $x$. We essentially require that $\kappa'$ is strictly positive, which is self evident, but also, that it is increasing. The latter means, that if our position $x$ is small, we will trade aggressively with low risk aversion, because we have nothing to lose, but if our position is large, for instance due to good profits in the past, we will prefer more conservative trading strategies to lock in the gains.

The fact that we require $U''$ to have the structure $-\exp(-\kappa(x))$ with some sufficiently smooth function $\kappa$ is not really restrictive due to the fact that it already follows from assuming that $U$ is strictly concave (it already has to be concave in order to be a utility function) and sufficiently smooth. Neither do we consider differentiability and boundedness assumptions related to $\kappa$ structurally restrictive. The only ''hard'' restrictions are $\kappa'\geq\varepsilon>0$ and $\kappa''\geq 0$ both of which make sense in general as we have motivated above.
\end{remark}

Let us prove some properties of utility functions $U$ with the above structure:

\begin{lemma}\label{Uproperties}
Assume that $U:\mathbb{R}\rightarrow\mathbb{R}$ satisfies condition (C1). Then the following holds:
\begin{itemize}
\item $U''<0$ everywhere,
\item $(\ln(-U''))''$ is non-positive and bounded,
\item $\frac{U'}{U''}$, $\frac{U''}{U'}$ and $\frac{U^{(3)}}{U''}$ are bounded.
\end{itemize}
\end{lemma}
\begin{proof}
Note that $\kappa'$ is non-decreasing and bounded, therefore $\lim_{x\rightarrow\infty}\kappa'(x)$ and $\lim_{x\rightarrow-\infty}\kappa'(x)$ exist. Also, $\lim_{x\rightarrow-\infty}\kappa'(x)\geq\varepsilon>0$. For very large or very small values $\kappa$ behaves linearly. It is also convex and strictly increasing. \\
We have
$$ U''(x)=-\exp(-\kappa(x)), $$
$$ U'(x)=\int_x^\infty \exp(-\kappa(y))\dx y. $$
Note that $\exp(-\kappa(x))\leq \exp(-\gamma x)$ for some fixed $\gamma>0$ and sufficiently large $x$. Therefore, the expression $\int_x^\infty \exp(-\kappa(y))\dx y$ is well-defined for all $x\in\mathbb{R}$ and bounded by $\frac{1}{\gamma}\exp(-\gamma x)$ for sufficiently large $x$. Therefore, the expression $$\int_x^\infty\int_y^\infty \exp(-\kappa(z))\dx z\dx y$$ is also well-defined. We observe furthermore:
\begin{itemize}
\item  $U'(x)>0$ for all $x\in\mathbb{R}$,
\item $U''(x)<0$ for all $x\in\mathbb{R}$,
\item $(\ln(-U''))''=-\kappa''(x)$ is non-positive and bounded.
\end{itemize}
\emph{We claim that $\frac{U'}{U''}$ is bounded:}

Clearly, $\frac{U'}{U''}<0$. Also, $\frac{U'}{U''}$ is continuous. Let us analyze its behaviour for $x\rightarrow\infty$ and $x\rightarrow-\infty$. We do this by L'H\^opital's rule: Clearly, $\lim_{x\rightarrow\infty} U''(x)=0$, $\lim_{x\rightarrow\infty} U'(x)=0$, $\lim_{x\rightarrow-\infty} U''(x)=-\infty$, $\lim_{x\rightarrow-\infty} U'(x)=\infty$, due to $\inf_{x\in\mathbb{R}}\kappa'(x)>0$. We have
$$ \lim_{x\rightarrow\infty}\frac{U'(x)}{U''(x)}=\lim_{x\rightarrow\infty}\frac{-\exp(-\kappa(x))}{\kappa'(x)\exp(-\kappa(x))}=
\lim_{x\rightarrow\infty}\frac{-1}{\kappa'(x)}=-\frac{1}{\lim_{x\rightarrow\infty}\kappa'(x)}>-\infty $$
and similarly
$$ \lim_{x\rightarrow-\infty}\frac{U'(x)}{U''(x)}=-\frac{1}{\lim_{x\rightarrow-\infty}\kappa'(x)}>-\infty. $$
Since $\frac{U'}{U''}$ is continuous $\inf_{x\in\mathbb{R}} \frac{U'}{U''}(x)>-\infty$ must hold and so $\frac{U'}{U''}$ is bounded.
\hfill\myqed
\vspace{1mm}\\
\emph{Also, we see that $\frac{U^{(3)}}{U''}$ is bounded:} $\frac{U^{(3)}}{U''}=\frac{\kappa'\exp(-\kappa)}{-\exp(-\kappa)}=-\kappa'$, which is bounded. \\
\vspace{1mm}
\emph{Finally we have to demonstrate that $\frac{U''}{U'}$ is bounded as well: }
\vspace{1mm} \\
This is again done by taking into account that $\frac{U''}{U'}$ is negative and continuous. Using L'H\^opital's rule:
$$ \lim_{x\rightarrow\infty}\frac{U''(x)}{U'(x)}=\lim_{x\rightarrow\infty}\frac{\kappa'(x)\exp(-\kappa(x))}{-\exp(-\kappa(x))}=-\lim_{x\rightarrow\infty}\kappa'(x)>-\infty, $$
$$ \lim_{x\rightarrow-\infty}\frac{U''(x)}{U'(x)}=\lim_{x\rightarrow-\infty}\frac{\kappa'(x)\exp(-\kappa(x))}{-\exp(-\kappa(x))}=-\lim_{x\rightarrow-\infty}\kappa'(x)>-\infty. $$
\end{proof}


In addition to condition (C1), we make the following assumption (C2) on $\theta$ and $H$:
\begin{description}
\item{(C2)} $\theta$ and $H$ depend on $\omega$ only through a standard, possibly high dimensional, diffusion. More precisely we assume that there is an $\mathbb{R}^{1\times N}$ - dimensional progressively measurable process $\tilde{X}$ on $[0,T]$ such that
$$ \tilde{X}_t=\tilde{x}+\int_{0}^t\tilde{\mu}(s,\tilde{X}_s)\dx s+\int_{0}^t\dx W^\top_s\tilde{\sigma}(s,\tilde{X}_s), $$
for some $\tilde{x}\in\mathbb{R}^{1\times N}$ and such that $H(X_T):=\tilde{H}(\tilde{X}_T,X_T)$ and $\theta_s:=\tilde\theta(s,\tilde{X}_s)$, for all $s\in[0,T]$, where 
\begin{itemize}
\item $\tilde{\mu},\tilde{\sigma}:[0,T]\times\mathbb{R}^{1\times N}\rightarrow\mathbb{R}^{1\times N},\mathbb{R}^{d\times N}$ are measurable, Lipschitz continuous in the second component with Lipschitz constants $L_{\tilde\mu,\tilde{x}}$, $L_{\tilde\sigma,\tilde{x}}$ and satisfy $\|\tilde\mu(\cdot,0)\|_\infty,\|\tilde\sigma\|_\infty<\infty$,
\item $\tilde\theta:[0,T]\times\mathbb{R}^{1\times N}\rightarrow\mathbb{R}^d$ is measurable, bounded and differentiable in the second component everywhere with a uniformly bounded derivative,
\item $\tilde H:\mathbb{R}^{1\times N}\times\mathbb{R}\rightarrow\mathbb{R}$ is bounded and Lipschitz continuous in both components with Lipschitz constants $L_{\tilde H,\tilde{x}}$, $L_{\tilde H,x}$, such that
\item $L_{\tilde H,x}<1$, where $x\in\mathbb{R}$ refers to the second component.
\end{itemize}
\end{description}

\begin{remark}
This assumption can be motivated by the following heuristic arguments:
\begin{itemize}
\item We suspect that it is possible to adequately approximate every ''non-pathological'' measurable $H\in L^\infty(\Omega,\mathcal{F}_T)$ by an $H$ with the above structural properties:
First approximate $H$ by a deterministic function of $(W_{t_i})_{i=1,\ldots,N}$ for finitely many times $t_i$ and then approximate every $W_{t_i}$ by the terminal value of a standard forward diffusion with vanishing drift and a volatility which assumes values between $0$ and $1$, is close to $1$ before time  $t_i$ and close to $0$ after that, such that $W_{t_i}$ is isolated.
\item A similar approximative argument could be applied to $\theta$.
\item In general, when trying to treat FBSDEs numerically assumptions which make the problem Markovian are usually made anyway (see e.g. \cite{benderdenk}).
\end{itemize}
\end{remark}

We show in section \ref{mainresult} that under assumptions (C1) and (C2) the problem \eqref{utility-FBSDE} has a unique solution $(X,Y,Z)$. As already mentioned this leads to an optimal strategy via \eqref{heyne-fbsde:opt:strat}. Consult \cite{HIRZ13}, Theorem 3.5. for the proof of optimality (note also Remark \ref{martingalerem} in this context). To show existence and uniqueness of $(X,Y,Z)$ in section \ref{mainresult} we use:

\section{The method of decoupling fields}\label{decf}

In this section we will briefly summarize the key results of the abstract theory of Markovian decoupling fields, we rely on later in the paper. The presented theory is derived from the SLC theory of Chapter 2 of \cite{Fromm2015} and is proven in \cite{Proemel2015}.

We consider families $(\mu,\sigma,f)$ of measurable functions, more precisely
$$ \mu: [0,T]\times\Omega\times\mathbb{R}^n\times\mathbb{R}^m\times\mathbb{R}^{m\times d}\longrightarrow \mathbb{R}^n, $$
$$ \sigma: [0,T]\times\Omega\times\mathbb{R}^n\times\mathbb{R}^m\times\mathbb{R}^{m\times d}\longrightarrow \mathbb{R}^{n\times d}, $$
$$ f: [0,T]\times\Omega\times\mathbb{R}^n\times\mathbb{R}^m\times\mathbb{R}^{m\times d}\longrightarrow \mathbb{R}^m, $$
where $n,m,d\in\mathbb{N}$ and $T>0$, $(\Omega,\mathcal{F},\mathbb{P},(\mathcal{F}_t)_{t\in[0,T]})$ is a complete filtered probability space,
such that $\mathcal{F}_0$ consists of all null sets, $\mathcal{F}=\mathcal{F}_T$ and
$\mathcal{F}_t=\sigma(\mathcal{F}_0,(W_s)_{s\in [0,t]})$ for all $t\in [0,T]$ holds,
 where $(W_t)_{t\in[0,T]}$ is a $d$-dimensional Brownian motion.

We want $\mu$, $\sigma$ and $f$ to be progressively measurable w.r.t. $(\mathcal{F}_t)_{t\in[0,T]}$, i.e. $\mu\mathbf{1}_{[0,t]},\sigma\mathbf{1}_{[0,t]},f\mathbf{1}_{[0,t]}$ must be $\mathcal{B}([0,T])\otimes\mathcal{F}_t\otimes\mathcal{B}(\mathbb{R}^n)\otimes\mathcal{B}(\mathbb{R}^m)\otimes\mathcal{B}(\mathbb{R}^{m\times d})$ - measurable for all $t\in[0,T]$. We will assume throughout the section that $\mu$, $\sigma$ and $f$ have this property without mentioning it.

A problem given by $\mu,\sigma,f,\xi$ is said to be \emph{Markovian}, if these four functions are deterministic, i.e. depend on $t,x,y,z$ only.
In the Markovian case we can somewhat relax the Lipschitz continuity assumptions of Chapter 2 of \cite{Fromm2015} and still obtain local existence together with uniqueness. What makes the Markovian case so special is the property
$$"Z_s=u_x(s,X_s)\cdot\sigma(s,X_s,Y_s,Z_s)"$$
which comes from the fact that $u$ will also be deterministic. This property allows us to bound $Z$ by a constant if we assume that $\sigma$ is bounded.

This potential boundedness of $Z$ in the Markovian case motivates the following definition, which will allows to develop a theory for non-Lipschitz problems:

\begin{defi}
Let $\xi:\Omega\times\mathbb{R}^n\rightarrow\mathbb{R}^m$ be measurable and let $t\in[0,T]$.\\
We call a function $u:[t,T]\times\Omega\times\mathbb{R}^n\rightarrow\mathbb{R}^m$ with $u(T,\omega,\cdot)=\xi(\omega,\cdot)$ for a.a. $\omega\in\Omega$ a \emph{Markovian decoupling field} for $\fbsde (\xi,(\mu,\sigma,f))$ on $[t,T]$ if for all $t_1,t_2\in[t,T]$ with $t_1\leq t_2$ and any $\mathcal{F}_{t_1}$ - measurable $X_{t_1}:\Omega\rightarrow\mathbb{R}^n$ there exist progressive processes $X,Y,Z$ on $[t_1,t_2]$ such that
\begin{itemize}
\item $X_s=X_{t_1}+\int_{t_1}^s\mu(r,X_r,Y_r,Z_r)\dx r+\int_{t_1}^s\sigma(r,X_r,Y_r,Z_r)\dx W_r$ a.s.,
\item $Y_s=Y_{t_2}-\int_{s}^{t_2}f(r,X_r,Y_r,Z_r)\dx r-\int_{s}^{t_2}Z_r\dx W_r$ a.s.,
\item $Y_s=u(s,X_s)$ a.s.
\end{itemize}
for all $s\in[t_1,t_2]$ \underline{and such that $\|Z\|_\infty<\infty$ holds}.\\ In particular, we want all integrals to be well-defined and $X,Y,Z$ to have values in $\mathbb{R}^n$, $\mathbb{R}^m$ and $\mathbb{R}^{m\times d}$ respectively. \\
Furthermore, we call a function $u:(t,T]\times\mathbb{R}^n\rightarrow\mathbb{R}^m$ a Markovian decoupling field for $\fbsde(\xi,(\mu,\sigma,f))$ on $(t,T]$ if $u$ restricted to $[t',T]$ is a Markovian decoupling field for all $t'\in(t,T]$.
\end{defi}

A Markovian decoupling field is always a decoupling field in the standard sense as well. The only difference between the two notions is that we are only interested in $X,Y,Z$, where $Z$ is a.s. bounded.\\
Regularity for Markovian decoupling fields is defined very similarly to standard regularity: 

We write $\mathbb{E}_{t,\infty}[X]$ for $\esssup\,\mathbb{E}[X|\mathcal{F}_t]$ in the following definition:

\begin{defi}
Let $u:[t,T]\times\Omega\times\mathbb{R}^n\rightarrow\mathbb{R}^m$ be a Markovian decoupling field to $\fbsde(\xi,(\mu,\sigma,f))$. We call $u$ \emph{weakly regular}, if
$L_{u,x}<L_{\sigma,z}^{-1}$ and $\sup_{s\in[t,T]}\|u(s,\cdot,0)\|_{\infty}<\infty$.

Furthermore, we call a weakly regular $u$ \emph{strongly regular} if for all fixed $t_1,t_2\in[t,T]$, $t_1\leq t_2,$ the processes $X,Y,Z$ arising in the defining property of a Markovian decoupling field
are a.e. unique for each \emph{constant} initial value $X_{t_1}=x\in\mathbb{R}^n$ and satisfy
\begin{equation}\label{STRongregul1}
\sup_{s\in [t_1,t_2]}\mathbb{E}_{t_1,\infty}[|X_s|^2]+\sup_{s\in [t_1,t_2]}\mathbb{E}_{t_1,\infty}[|Y_s|^2]
+\mathbb{E}_{t_1,\infty}\left[\int_{t_1}^{t_2}|Z_s|^2\dx s\right]<\infty\quad\forall x\in\mathbb{R}^n.
\end{equation}
 In addition they must be measurable as functions of $(x,s,\omega)$ and
even weakly differentiable w.r.t. $x\in\mathbb{R}^n$ such that for every $s\in[t_1,t_2]$ the mappings $X_s$ and $Y_s$ are measurable functions of $(x,\omega)$ and even weakly differentiable w.r.t. $x$ such that
\begin{multline}\label{STRongregul2}
\esssup_{x\in\mathbb{R}^n}\sup_{v\in S^{n-1}}\sup_{s\in [t_1,t_2]}\mathbb{E}_{t_1,\infty}\left[\left|\frac{\dx}{\dx x}X_s\right|^2_v\right]<\infty, \\
\esssup_{x\in\mathbb{R}^n}\sup_{v\in S^{n-1}}\sup_{s\in [t_1,t_2]}\mathbb{E}_{t_1,\infty}\left[\left|\frac{\dx}{\dx x}Y_s\right|^2_v\right]<\infty, \\
\esssup_{x\in\mathbb{R}^n}\sup_{v\in S^{n-1}}\mathbb{E}_{t_1,\infty}\left[\int_{t_1}^{t_2}\left|\frac{\dx}{\dx x}Z_s\right|^2_v\dx s\right]<\infty,
\end{multline}
where $S^{n-1}$ is the $(n-1)$ - dimensional sphere. \\
We say that a Markovian decoupling field $u$ on $[t,T]$ is \emph{strongly regular} on a subinterval $[t_1,t_2]\subseteq[t,T]$ if $u$ restricted to $[t_1,t_2]$ is a strongly regular Markovian decoupling field for $\fbsde(u(t_2,\cdot),(\mu,\sigma,f))$. \\
Furthermore, we say that a Markovian decoupling field $u:(t,T]\times\mathbb{R}^n\rightarrow\mathbb{R}^m$
\begin{itemize}
\item is weakly regular if $u$ restricted to $[t',T]$ is weakly regular for all $t'\in(t,T]$,
\item is strongly regular if $u$ restricted to $[t',T]$ is strongly regular for all $t'\in(t,T]$.
\end{itemize}
\end{defi}
For the following class of problems an existence and uniqueness theory is developed:
\begin{defi}
We say that $\xi,\mu,\sigma,f$ satisfy \emph{modified local Lipschitz conditions (MLLC)} if
\begin{itemize}
\item $\mu,\sigma,f$ are
\begin{itemize}
\item deterministic,
\item Lipschitz continuous in $x,y,z$ on sets of the form $[0,T]\times\mathbb{R}^n\times\mathbb{R}^{m} \times B$, where $B\subset \mathbb{R}^{m\times d}$ is an arbitrary bounded set
\item and such that $\|\mu(\cdot,0,0,0)\|_\infty,\|f(\cdot,0,0,0)\|_{\infty},\|\sigma(\cdot,\cdot,\cdot,0)\|_{\infty},L_{\sigma,z}<\infty$,
\end{itemize}
\item $\xi: \mathbb{R}^n\rightarrow \mathbb{R}^m$ satisfies $L_{\xi,x}<L_{\sigma,z}^{-1}$,
\end{itemize}
where $L_{\sigma,z}$ denotes the Lipschitz constant of $\sigma$ w.r.t. the dependence on the last component $z$ (and w.r.t. the Frobenius norms on $\mathbb{R}^{m\times d}$ and $\mathbb{R}^{n\times d}$).
By $L_{\sigma,z}^{-1}=\frac{1}{L_{\sigma,z}}$ we mean $\frac{1}{L_{\sigma,z}}$ if $L_{\sigma,z}>0$ and $\infty$ otherwise.
\end{defi}

The following natural concept introduces a type of Markovian decoupling field for non-Lipschitz problems (non-Lipschitz in $z$), to which nevertheless standard Lipschitz results can be applied.

\begin{defi}
Let $u$ be a Markovian decoupling field for $\fbsde(\xi,(\mu,\sigma,f))$. We call $u$ \emph{controlled in $z$} if there exists a constant $C>0$ such that for all $t_1,t_2\in[t,T]$, $t_1\leq t_2$, and all initial values $X_{t_1}$, the corresponding processes $X,Y,Z$ from the definition of a Markovian decoupling field satisfy $|Z_s(\omega)|\leq C$,
for almost all $(s,\omega)\in[t,T]\times\Omega$. If for a fixed triple $(t_1,t_2,X_{t_1})$ there are different choices for $X,Y,Z$, then all of them are supposed to satisfy the above control.

We say that a Markovian decoupling field $u$ on $[t,T]$ is \emph{controlled in $z$} on a subinterval $[t_1,t_2]\subseteq[t,T]$ if $u$ restricted to $[t_1,t_2]$ is a Markovian decoupling field for $\fbsde(u(t_2,\cdot),(\mu,\sigma,f))$ that is controlled in $z$.

Furthermore, we call a Markovian decoupling field on an interval $(s,T]$ \emph{controlled in $z$} if it is controlled in $z$ on every compact subinterval $[t,T]\subseteq (s,T]$ (with $C$ possibly depending on $t$).
\end{defi}

The following important result allows us to connect the MLLC - case to SLC.

\begin{thm}[Theorem 3.16 in \cite{Proemel2015}.]\label{CONtrollM}
Let $\mu,\sigma,f,\xi$ satisfy MLLC and assume that there exists a weakly regular Markovian decoupling field $u$ to this problem on some interval $[t,T]$. Then $u$ is controlled in $z$.
\end{thm}

As applications of analogous results for SLC problems from Chapter 2 of \cite{Fromm2015} one shows

\begin{thm}[Theorem 3.17 in \cite{Proemel2015}.]\label{UNIqMREGulM}$\,$
\begin{enumerate}
\item Let $\mu,\sigma,f,\xi$ satisfy MLLC and assume that there are two weakly regular Markovian decoupling fields $u^{(1)},u^{(2)}$ to this problem on some interval $[t,T]$.
Then $u^{(1)}=u^{(2)}$ (up to modifications).
\item Let $\mu,\sigma,f,\xi$ satisfy MLLC and assume that there exists a weakly regular Markovian decoupling field $u$ to this problem on some interval $[t,T]$. Then $u$ is strongly regular.
\end{enumerate}
\end{thm}

Existence of weakly regular decoupling fields implies existence and uniqueness of classical solutions:

\begin{lemma}[Theorem 3.18 in \cite{Proemel2015}.]\label{UNIqXYZM}
Let $\mu,\sigma,f,\xi$ satisfy MLLC and assume that there exists a weakly regular Markovian decoupling field $u$ on some interval $[t,T]$.\\ Then for any initial condition $X_t=x\in\mathbb{R}^n$ there is a unique solution $(X,Y,Z)$ of the FBSDE on $[t,T]$ such that
$$\sup_{s\in[t,T]}\mathbb{E}[|X_s|^2]+\sup_{s\in[t,T]}\mathbb{E}[|Y_s|^2]+\|Z\|_\infty<\infty.$$
\end{lemma}

\begin{defi}
Let $I^{M}_{\mathrm{max}}\subseteq[0,T]$ for $\fbsde(\xi,(\mu,\sigma,f))$ be the union of all intervals $[t,T]\subseteq[0,T]$ such that there exists a weakly regular Markovian decoupling field $u$ on $[t,T]$.
\end{defi}

\begin{thm}[Global existence in weak form, Theorem 3.21 in \cite{Proemel2015}.]\label{GLObalexistM}
Let $\mu,\sigma,f,\xi$ satisfy MLLC. Then there exists a unique weakly regular Markovian decoupling field $u$ on $I^M_{\mathrm{max}}$. This $u$ is also controlled in $z$, strongly regular, deterministic and continuous. \\
Furthermore, either $I^{M}_{\mathrm{max}}=[0,T]$ or $I^{M}_{\mathrm{max}}=(t^{M}_{\mathrm{min}},T]$, where $0\leq t^{M}_{\mathrm{min}}<T$.
\end{thm}

The following result basically states that for a singularity $t^{M}_{\mathrm{min}}$ to occur $u_x$ has to "explode" at $t^M_{\mathrm{min}}$. It is the key to showing well-posedness for particular problems via contradiction.

\begin{lemma}[Lemma 3.22 in \cite{Proemel2015}.]\label{EXPlosionM}
Let $\mu,\sigma,f,\xi$ satisfy MLLC. If $I^{M}_{\mathrm{max}}=(t^{M}_{\mathrm{min}},T]$, then
$$\lim_{t\downarrow t^{M}_{\mathrm{min}}}L_{u(t,\cdot),x}=L_{\sigma,z}^{-1},$$
where $u$ is the unique weakly regular Markovian decoupling field from Theorem \ref{GLObalexistM}.
\end{lemma}

\section{Solving the FBSDE}\label{solvingthefbsde}

Before showing well-posedness of \eqref{utility-FBSDE} we first prove the more abstract Theorem \ref{sc3}:

\subsection{An abstract result}

For some $\varepsilon>0$ consider a forward-backward system of the form
\begin{multline}\label{epssystem}
\tilde{X}_s=\tilde{x}+\int_{t}^s\frac{1}{\varepsilon}\tilde{\mu}(r,\varepsilon\tilde{X}_r)\dx r+\int_{t}^s\dx W^\top_r\frac{1}{\varepsilon}\tilde{\sigma}(r,\varepsilon\tilde{X}_r), \\
\overline{X}_s=\overline{x}+\int_{t}^s\overline{\mu}(r,\varepsilon\tilde{X}_r)\dx r+\int_{t}^s\dx W^\top_r\overline{\sigma}(r,\varepsilon\tilde{X}_r,\overline{X}_r,Y_r,Z_r), \\
Y_s=\xi(\varepsilon\tilde{X}_T,\overline{X}_T)-\int_s^T f(r,\varepsilon\tilde{X}_r,\overline{X}_r,Y_r,Z_r)\dx r-\int_s^T\dx W^\top_r Z_r
\end{multline}
a.s. for all $s\in [t,T]$, where $\tilde{X}$ is $N$-dimensional, $N\in\mathbb{N}$, and $\overline{X},Y$ are real-valued. We assume that
\begin{itemize}
\item $\tilde{\mu},\tilde{\sigma},\overline{\mu}:[0,T]\times\mathbb{R}^{1\times N}\rightarrow\mathbb{R}^{1\times N},\mathbb{R}^{d\times N},\mathbb{R}$ are measurable and Lipschitz continuous in the second component with Lipschitz constants $L_{\tilde\mu,\tilde{x}}$, $L_{\tilde\sigma,\tilde{x}}$, $L_{\overline\mu,\tilde{x}}$ and such that 
$$\|\tilde\mu(\cdot,0)\|_\infty, \|\tilde\sigma\|_\infty, \|\overline\mu(\cdot,0)\|_\infty<\infty,$$
\item $\overline{\sigma}:[0,T]\times\mathbb{R}^{1\times N}\times\mathbb{R}\times\mathbb{R}\times\mathbb{R}^d\rightarrow\mathbb{R}^d$ is measurable, Lipschitz continuous in the last four components and satisfies $\|\overline{\sigma}(\cdot,\cdot,\cdot,\cdot,0)\|_\infty<\infty$,
\item $f:[0,T]\times\mathbb{R}^{1\times N}\times\mathbb{R}\times\mathbb{R}\times\mathbb{R}^d\rightarrow\mathbb{R}$ is measurable and Lipschitz continuous in the last four components on sets of the form $[0,T]\times\mathbb{R}^{1\times N}\times\mathbb{R}\times\mathbb{R}\times B$, where $B\subseteq \mathbb{R}^d$ is bounded. We also assume $\|f(\cdot,0,0,0,0)\|_\infty<\infty$.
\item $\xi:\mathbb{R}^{1\times N}\times\mathbb{R}\rightarrow\mathbb{R}$ is Lipschitz continuous in both components with the two Lipschitz constants $L_{\xi,\tilde{x}}$ and $L_{\xi,\overline{x}}$. We assume that $L_{\xi,\overline{x}}<L_{\overline{\sigma},z}^{-1}$, where $L_{\overline{\sigma},z}$ refers to the Lipschitz constant of $\overline{\sigma}$ w.r.t. the last component. Furthermore, $L_{\xi}$ refers to the Lipschitz constant of $\xi$ w.r.t. the Euclidian norm on $\mathbb{R}^{1\times N}\times\mathbb{R}$.
\end{itemize}
The problem is to find progressively measurable processes $\tilde{X},\overline{X},Y,Z$ s.t. $\tilde{X}$ is $\mathbb{R}^{1\times N}$ - valued, $\overline{X}$ and $Y$ are both $\mathbb{R}$ - valued, $Z$ is $\mathbb{R}^{d}$ - valued and such that \eqref{epssystem} is satisfied.

Note that for varying $\varepsilon>0$ the different problems are equivalent to each other in the following sense: If $\tilde{X}^{\varepsilon_1},\overline{X},Y,Z$ solve (\ref{epssystem}) for some $\varepsilon_1>0$ on some interval $[t,T]$, then $\tilde{X}^{\varepsilon_2}:=\frac{\varepsilon_1}{\varepsilon_2}\tilde{X}^{\varepsilon_1},\overline{X},Y,Z$ solve (\ref{epssystem}) for some $\varepsilon_2>0$ on the same interval. This means that we can choose the parameter $\varepsilon>0$ as we like without changing the nature of the problem. In particular, if we define the terminal condition $\xi^{\varepsilon}$ via $\xi^{\varepsilon}(\tilde{x},\overline{x}):=\xi(\varepsilon\tilde{x},\overline{x})$, we can ensure that the Lipschitz constant $L_{\xi^{\varepsilon}}$ of $\xi^{\varepsilon}$ satisfies $L_{\xi^{\varepsilon}}<L_{\overline{\sigma},z}$ by choosing $\varepsilon$ small enough! This explains why we work with the parameter $\varepsilon>0$.

Also, note that (\ref{epssystem}) describes a Markovian problem, which satisfies MLLC (for $\varepsilon$ small enough), such that the theory previously described is well applicable: The forward equation is $(N+1)$ - dimensional and the backward equation has dimension $1$. Also, observe that the first $N$ components of the forward equation do not depend on the rest of the problem, i.e. $\tilde{X}$ depends only on the parameters $\tilde{x},\tilde{\mu},\tilde{\sigma}$ and $\varepsilon$.

Assume that we have $d_1\in\mathbb{N}$, $d_2\in\mathbb{N}_0$ such that $d_1+d_2=d$, which is the dimension of our Brownian motion $W$. We denote by $\pi_1:\mathbb{R}^d\rightarrow\mathbb{R}^d$ the linear operator which sets the last $d_2$ components of a vector to zero and leaves the first $d_1$ unchanged. Similarly, $\pi_2$ is the operator which sets the first $d_1$ components of a vector in $\mathbb{R}^d$ to zero without changing the others.\\
To be able to treat the above MLLC problem we make the following structural requirements for $f$:
\begin{itemize}
\item $f$ can be written as a function of $t,\tilde{x},\overline{x}+y,\pi_2(z)$,
\item $f$ is (classically) differentiable in $(\tilde{x},\overline{x}+y,z)$ everywhere with $\frac{\dx}{\dx (\overline{x}+y)}f\geq 0$,
\item $|\frac{\dx}{\dx z}f(s,\tilde{x},\overline{x}+y,z)|\leq C(1+|z|)$ for all $s,\tilde{x},\overline{x},y,z$ with some constant $C>0$,
\item $|\frac{\dx}{\dx (\overline{x}+y)}f(s,\tilde{x},\overline{x}+y,z)|\leq C(1+|z|^2)$ for all $s,\tilde{x},\overline{x},y,z$ with some constant $C>0$,
\item $\|f(\cdot,\cdot,\cdot,\cdot,0)\|_\infty<\infty$,
\item $\left\|\frac{\dx}{\dx\tilde{x}}f\right\|_{\infty}<\infty$.
\end{itemize}
In addition, we make the following structural assumptions for $\overline{\sigma}$ and $\xi$:
\begin{itemize}
\item $\overline{\sigma}$ has the form $\overline{\sigma}=\VV{\overline{\sigma}^{(1)}}{\overline{\sigma}^{(2)}}$, with a $d_1$ - dimensional $\overline{\sigma}^{(1)}$ and $d_2$ - dimensional $\overline{\sigma}^{(2)}$, such that $\overline{\sigma}^{(1)}$ is a function of $t,\tilde{x},\overline{x},y,\pi_1(z)$ and $\overline\sigma^{(2)}$ is a function of $t,\tilde{x},\pi_2(z)$,
\item $\overline{\sigma}$ is differentiable in $(\tilde{x},\overline{x},y,z)$ everywhere with bounded derivatives,
\item $L_{\xi,\overline{x}}<1\leq L_{\overline{\sigma},z}^{-1}$,
\item $\|\xi\|_\infty<\infty$.
\end{itemize}
Under these conditions we can prove the following abstract result, which will be applied to the particular FBSDE later on. It basically states that for $I^{M}_{\mathrm{max}}=[0,T]$ to hold it is enough to control the Lipschitz constant of $u$ w.r.t. $\overline{x}\in\mathbb{R}$: The Lipschitz constant w.r.t. $\tilde{x}\in\mathbb{R}^{1\times N}$ is then controlled automatically as well.

\begin{thm}\label{sc3} Assume that the above problem has the following property: Every weakly regular Markovian decoupling field $u:[t,T]\times\mathbb{R}^{1\times N}\times\mathbb{R}\rightarrow\mathbb{R}$ satisfies
$$\sup_{s\in[t,T]}\left\|\frac{\dx}{\dx \overline{x}}u(s,\cdot,\cdot)\right\|_\infty\leq K$$
for all $t$, $u$, $\varepsilon\in(0,\varepsilon_0]$, where $K<L_{\overline{\sigma},z}^{-1}$ and $\varepsilon_0>0$.\\
Then there exists an $\varepsilon>0$ such that for the above problem $I^{M}_{\mathrm{max}}=[0,T]$ holds true.
\end{thm}
\begin{proof}
Assume $I^{M}_{\mathrm{max}}=(t^{M}_{\mathrm{min}},T]$ with some $t^{M}_{\mathrm{min}}\in [0,T)$. Let from now on $u$ be the weakly regular Markovian decoupling field from Theorem \ref{GLObalexistM} defined on the whole of $I^{M}_{\mathrm{max}}$. 
Note also that $u$ is strongly regular.

Choose any $t_1\in(t^M_{\mathrm{min}},T]$, any $\tilde{x}\in\mathbb{R}^N,\overline{x}\in\mathbb{R}$ and consider the corresponding FBSDE on $[t_1,T]$:
\begin{itemize}
\item $\tilde{X}_s=\tilde{x}+\int_{t_1}^s\frac{1}{\varepsilon}\tilde{\mu}(r,\varepsilon\tilde{X}_r)\dx r+\int_{t_1}^s\dx W^\top_r\frac{1}{\varepsilon}\tilde{\sigma}(r,\varepsilon\tilde{X}_r)$,
\item $\overline{X}_s=\overline{x}+\int_{t_1}^s\overline{\mu}(r,\varepsilon\tilde{X}_r)\dx r+\int_{t_1}^s\dx W^\top_r\overline{\sigma}(r,\varepsilon\tilde{X}_r,\overline{X}_r,Y_r,Z_r)$,
\item $Y_s=\xi(\varepsilon\tilde{X}_T,\overline{X}_T)-\int_s^T f(r,\varepsilon\tilde{X}_r,\overline{X}_r,Y_r,Z_r)\dx r-\int_s^T\dx W^\top_r Z_r$
\end{itemize}
where $s\in[t_1,T]$.
We now differentiate w.r.t. $\tilde{x}$ and $\overline{x}$ using strong regularity and the chain rule of Lemma A.3.1. in \cite{Fromm2015}:
\begin{equation}\label{tildetildex}
\frac{\dx}{\dx \tilde{x}}\tilde{X}_s=I_N+\int_{t_1}^s\delta^{\tilde{\mu},\tilde{x}}_r\frac{\dx}{\dx \tilde{x}}\tilde{X}_r\dx r+\int_{t_1}^s\dx W^\top_r\delta^{\tilde{\sigma},\tilde{x}}_r\frac{\dx}{\dx \tilde{x}}\tilde{X}_r,
\end{equation}
\begin{equation}\label{tildeoverx}
\frac{\dx}{\dx \tilde{x}}\overline{X}_s=\int_{t_1}^s\varepsilon\delta^{\overline{\mu},\tilde{x}}_r\frac{\dx}{\dx \tilde{x}}\tilde{X}_r\dx r
+\int_{t_1}^s\dx W^\top_r\left(\varepsilon\delta^{\overline{\sigma},\tilde{x}}_r\frac{\dx}{\dx \tilde{x}}\tilde{X}_r+\delta^{\overline{\sigma},\overline{x}}_r\frac{\dx}{\dx \tilde{x}}\overline{X}_r+
\delta^{\overline{\sigma},y}_r\frac{\dx}{\dx \tilde{x}}Y_r+\delta^{\overline{\sigma},z}_r\frac{\dx}{\dx \tilde{x}}Z_r\right),
\end{equation}
\begin{multline}\label{tildey}
\frac{\dx}{\dx \tilde{x}}Y_s=\frac{\dx}{\dx \tilde{x}}Y_{T}-\int_{s}^{T}\dx W^\top_r\frac{\dx}{\dx \tilde{x}}Z_r- \\
-\int_{s}^{T}\left(\varepsilon\delta^{f,\tilde{x}}_r\frac{\dx}{\dx \tilde{x}}\tilde{X}_r+\delta^{f,\overline{x}+y}_r\left(\frac{\dx}{\dx \tilde{x}}\overline{X}_r+\frac{\dx}{\dx \tilde{x}}Y_r\right)+\delta^{f,z}_r\frac{\dx}{\dx \tilde{x}}Z_r\right)\dx r,
\end{multline}
\begin{equation}\label{overoverx}
\frac{\dx}{\dx \overline{x}}\overline{X}_s=1+\int_{t_1}^s\dx W^\top_r\left(\delta^{\overline{\sigma},\overline{x}}_r\frac{\dx}{\dx \overline{x}}\overline{X}_r+
\delta^{\overline{\sigma},y}_r\frac{\dx}{\dx \overline{x}}Y_r+\delta^{\overline{\sigma},z}_r\frac{\dx}{\dx \overline{x}}Z_r\right),
\end{equation}
\begin{equation}\label{overy}
\frac{\dx}{\dx \overline{x}}Y_s=\frac{\dx}{\dx \overline{x}}Y_{T}-\int_{s}^{T}\dx W^\top_r\frac{\dx}{\dx \overline{x}}Z_r-
\int_{s}^{T}\left(\delta^{f,\overline{x}+y}_r\left(\frac{\dx}{\dx \overline{x}}\overline{X}_r+\frac{\dx}{\dx \overline{x}}Y_r\right)+\delta^{f,z}_r\frac{\dx}{\dx \overline{x}}Z_r\right)\dx r,
\end{equation}
a.s. for all $s\in[t_1,T]$, for almost all $(\tilde{x},\overline{x})\in\mathbb{R}^{1\times N}\times\mathbb{R}$ with the following progressively measurable and bounded processes:
\begin{itemize}
\item $\delta^{\tilde{\mu},\tilde{x}}$, $\delta^{\tilde{\sigma},\tilde{x}}$, $\delta^{\overline{\mu},\tilde{x}}$ which are provided by Lemma A.3.1. in \cite{Fromm2015} and are bounded independently of $t_1$ and $\varepsilon$,
\item $\delta^{\overline{\sigma},\tilde{x}}=\frac{\dx}{\dx\tilde{x}}\overline{\sigma}(\cdot,\tilde{X},\overline{X},Y,Z)$, $\delta^{f,\tilde{x}}=\frac{\dx}{\dx\tilde{x}}f(\cdot,\tilde{X},\overline{X},Y,Z)$, which are also uniformly bounded,
\item $\delta^{\overline{\sigma},\overline{x}}=\frac{\dx}{\dx\overline{x}}\overline{\sigma}(\cdot,\tilde{X},\overline{X},Y,Z)$, $\delta^{\overline{\sigma},y}=\frac{\dx}{\dx y}\overline{\sigma}(\cdot,\tilde{X},\overline{X},Y,Z)$ and $\delta^{\overline{\sigma},z}=\frac{\dx}{\dx z}\overline{\sigma}(\cdot,\tilde{X},\overline{X},Y,Z)$, which are uniformly bounded as well,
\item $\delta^{f,\overline{x}+y}=\frac{\dx}{\dx (\overline{x}+y)}f(\cdot,\tilde{X},\overline{X},Y,Z)$ and $\delta^{f,z}=\frac{\dx}{\dx z}f(\cdot,\tilde{X},\overline{X},Y,Z)$, which are bounded but not necessarily uniformly in $t_1$, $\varepsilon$.
\end{itemize}
Uniform boundedness of $\delta^{\tilde{\mu},\tilde{x}}$, $\delta^{\tilde{\sigma},\tilde{x}}$, $\delta^{\overline{\mu},\tilde{x}}$, $\delta^{\overline{\sigma},\tilde{x}}$, $\delta^{f,\tilde{x}}$, $\delta^{\overline{\sigma},\overline{x}}$, $\delta^{\overline{\sigma},y}$, $\delta^{\overline{\sigma},z}$ is a consequence of the Lipschitz continuity assumptions we have made. Boundedness of $\delta^{f,\overline{x}+y}$ and $\delta^{f,z}$, however, follows from the structural assumptions on $f$ together with the boundedness of $Z$. \\
Note also that:
\begin{itemize}
\item $\delta^{\tilde{\sigma},\tilde{x}}$ is $\mathbb{R}^{\left(d\times N\right)\times (1\times N)}$ - valued and can also be interpreted as a vector $\left(\delta^{\tilde{\sigma},\tilde{x},i}\right)_{i=1,\ldots,d}$, where $\delta^{\tilde{\sigma},\tilde{x},i}$ are $\mathbb{R}^{(1\times N)\times (1\times N)}$ - valued.
\end{itemize}
According to the structural assumptions for $f$ we have:
\begin{itemize}
\item $\delta^{f,\overline{x}+y}$ is real-valued and non-negative,
\item $\delta^{f,z}$ is an $\mathbb{R}^{1\times d}$ - valued vector, where the first $d_1$ components vanish.
\end{itemize}
According to the structural assumptions for $\overline\sigma$ we have:
\begin{itemize}
\item $\delta^{\overline{\sigma},z}$ is an $\mathbb{R}^{d\times d}$ - valued bounded matrix consisting of an upper left $d_1\times d_1$ block and a lower right $d_2\times d_2$ block, such that all remaining components vanish and the operator norm of the matrix itself is bounded by $L_{\overline{\sigma},z}\leq 1$,
\item $\delta^{\overline{\sigma},\overline{x}}$, $\delta^{\overline{\sigma},y}$ are $\mathbb{R}^{d}$ - valued vectors, for which the last $d_2$ components vanish.
\end{itemize}
Now, define
\begin{itemize}
\item $\check{U}_r:=\frac{\dx}{\dx \overline{x}}\overline{X}_r$, $V_r:=\frac{\dx}{\dx \overline{x}}Y_r$, $\tilde{Z}_r:=\frac{\dx}{\dx \overline{x}}Z_r$,
\item $\hat{V}_r:=\frac{\dx}{\dx\overline{x}}u(r,\tilde{X}_r,\overline{X}_r)$.
\end{itemize}
$\hat{V}$ is bounded by $K<L_{\overline\sigma,z}^{-1}$: We assume without loss of generality that $\left|\frac{\dx}{\dx x}u\right|\leq K$ everywhere.

Define a process $\hat{V}$ via $\hat{V}_r:=\frac{\dx}{\dx x}u(r,X_r)$, $r\in[t_1,T]$. We can assume without loss of generality that $\hat{V}$ is bounded, since $u$ is Lipschitz continuous in $x$ on every compact interval and we can assume that $\frac{\dx}{\dx x}u(t,\cdot)$ is uniformly bounded by $L_{u(t,\cdot),x}$ for every $t\in I_{\mathrm{max}}$.\\
Let $\tau\in [t_1,T]$ be any stopping time such that $U$ is positive on $[t_1,\tau]$. We will argue later that we can choose $\tau=T$. \\
For $s\in [t_1,\tau]$ we have, using the chain rule of Lemma A.3.1. in \cite{Fromm2015}, the relationship $V_s=\hat{V}_s\check{U}_s$. In particular, $\hat{V}_s\mathbf{1}_{[t_1,\tau]}=\frac{V_s}{\check{U}_s}\mathbf{1}_{[t_1,\tau]}$ and $\hat{V}$ has a modification, which is continuous on $[t_1,\tau]$.
Defining $\hat{U}:=\check{U}^{-1}=\frac{1}{\check{U}}$ on $[t_1,\tau]$, we can rewrite (\ref{overoverx}) as
$$\check{U}_{s\wedge\tau}=1+\int_{t_1}^{s\wedge\tau}\dx W_r^\top\left(\delta^{\overline{\sigma},\overline{x}}_r+\delta^{\overline{\sigma},y}_r\hat{V}_r+\delta^{\overline{\sigma},z}_r\tilde{Z}_r\hat{U}_r\right)\check{U}_r.$$
Applying the It\^o formula to $\hat{U}=\check{U}^{-1}=\frac{1}{\check{U}}$ we obtain
\begin{multline*}
\hat{U}_{s\wedge\tau}=1+\int_{t_1}^{s\wedge\tau}\hat{U}_r\bigg(\left(\delta^{\overline{\sigma},\overline{x}}_r+\delta^{\overline{\sigma},y}_r\hat{V}_r+\delta^{\overline{\sigma},z}_r\tilde{Z}_r\hat{U}_r\right)^\top\left(\delta^{\overline{\sigma},\overline{x}}_r+\delta^{\overline{\sigma},y}_r\hat{V}_r+\delta^{\overline{\sigma},z}_r\tilde{Z}_r\hat{U}_r\right)\bigg) \dx r- \\
-\int_{t_1}^{s\wedge\tau}\dx W_r^\top\hat{U}_r\left(\delta^{\overline{\sigma},\overline{x}}_r+\delta^{\overline{\sigma},y}_r\hat{V}_r+\delta^{\overline{\sigma},z}_r\tilde{Z}_r\hat{U}_r\right).
\end{multline*}
Applying the It\^o formula to $\hat{V}=V\hat{U}$ we get
\begin{multline*}
\hat{V}_{s\wedge\tau}=\hat{V}_{\tau}-\int_{s\wedge\tau}^{\tau}\dx W_r^\top\left(\tilde{Z}_r\hat{U}_r-\hat{V}_r\left(\delta^{\overline{\sigma},\overline{x}}_r+\delta^{\overline{\sigma},y}_r\hat{V}_r+\delta^{\overline{\sigma},z}_r\tilde{Z}_r\hat{U}_r\right)\right)-\\
-\int_{s\wedge\tau}^{\tau}\bigg(\delta^{f,\overline{x}+y}_r+\delta^{f,\overline{x}+y}_r\hat{V}_r+\delta^{f,z}_r\tilde{Z}_r\hat{U}_r+\\
+\hat{V}_r\left(\left(\mybox{\delta^{\overline{\sigma},\overline{x}}_r+\delta^{\overline{\sigma},y}_r\hat{V}_r+\delta^{\overline{\sigma},z}_r\tilde{Z}_r\hat{U}_r}\right)^\top\left(\delta^{\overline{\sigma},\overline{x}}_r+\delta^{\overline{\sigma},y}_r\hat{V}_r+\delta^{\overline{\sigma},z}_r\tilde{Z}_r\hat{U}_r\right)\right)- \\
-\mybox{\tilde{Z}_r^\top\hat{U}_r}\left(\delta^{\overline{\sigma},\overline{x}}_r+\delta^{\overline{\sigma},y}_r\hat{V}_r+\delta^{\overline{\sigma},z}_r\tilde{Z}_r\hat{U}_r\right)\bigg)\dx r.
\end{multline*}
Note the two marked terms above. Define a process $\hat{Z}$ on $[t_1,\tau]$ via
$$\hat{Z}_r:=\hat{U}_r\tilde{Z}_r-\hat{V}_r\left(\delta^{\overline{\sigma},\overline{x}}_r+\delta^{\overline{\sigma},y}_r\hat{V}_r+\delta^{\overline{\sigma},z}_r\tilde{Z}_r\hat{U}_r\right),\quad r\in [t_1,\tau].$$
Then we get
\begin{multline}\label{BS1}
\hat{V}_{s\wedge\tau}=\hat{V}_{\tau}-\int_{s\wedge\tau}^{\tau}\dx W_r^\top\hat{Z}_r-
\int_{s\wedge\tau}^{\tau}\bigg(\delta^{f,\overline{x}+y}_r+\delta^{f,\overline{x}+y}_r\hat{V}_r+\mybox{\delta^{f,z}_r\tilde{Z}_r\hat{U}_r}- \\
-\hat{Z}_r^\top\left(\delta^{\overline{\sigma},\overline{x}}_r+\delta^{\overline{\sigma},y}_r\hat{V}_r+\mybox{\delta^{\overline{\sigma},z}_r\tilde{Z}_r\hat{U}_r}\right)\bigg)\dx r.
\end{multline}
Note the terms $\delta^{\overline{\sigma},z}_r\tilde{Z}_r\hat{U}_r$ and $\delta^{f,z}_r\tilde{Z}_r\hat{U}_r$.
A straightforward calculation starting from the definition of $\hat{Z}_r$ leads to
$$\hat{Z}_r+\hat{V}_r\left(\delta^{\overline{\sigma},\overline{x}}_r+\delta^{\overline{\sigma},y}_r\hat{V}_r\right)=\tilde{Z}_r\hat{U}_r-\hat{V}_r\delta^{\overline{\sigma},z}_r\tilde{Z}_r\hat{U}_r=\left(I_d-\hat{V}_r\delta^{\overline{\sigma},z}_r\right)\tilde{Z}_r\hat{U}_r,$$
\begin{equation}\label{BS2}\tilde{Z}_r\hat{U}_r=
\left(I_{d}-\hat{V}_r\delta^{\overline{\sigma},z}_r\right)^{-1}\left(\hat{Z}_r+\hat{V}_r\left(\delta^{\overline{\sigma},\overline{x}}_r+\delta^{\overline{\sigma},y}_r\hat{V}_r\right)\right),\quad r\in [t_1,\tau],
\end{equation}
where $I_d\in\mathbb{R}^{d\times d}$ is the identity. \\
Note here that $\left\|\hat{V}\right\|_\infty\leq L_{u_{|[t_1,T]},x}<L_{\overline\sigma,z}^{-1}$ and also that the operator norm of $\delta^{\overline\sigma,z}$ is universally bounded by $L_{\overline\sigma,z}$, so the essential supremum of the operator norm of $\hat{V}_r\delta^{\overline{\sigma},z}_r$ is strictly smaller than $1$ and, therefore, the expression $\left(I_{d}-\hat{V}_r\delta^{\overline{\sigma},z}_r\right)^{-1}$ is well-defined and even universally bounded (on $[t_1,T]$) in the operator norm.  \\
By plugging (\ref{BS2}) into (\ref{BS1}) we obtain:
\begin{multline*}
\hat{V}_{s\wedge\tau}=\hat{V}_{\tau}-\int_{s\wedge\tau}^{\tau}\dx W_r^\top\hat{Z}_r- \\
-\int_{s\wedge\tau}^{\tau}\bigg(\delta^{f,\overline{x}+y}_r+\delta^{f,\overline{x}+y}_r\hat{V}_r+\delta^{f,z}_r\left(I_{d}-\hat{V}_r\delta^{\overline{\sigma},z}_r\right)^{-1}
\left(\hat{Z}_r+\hat{V}_r\left(\mybox{\delta^{\overline{\sigma},\overline{x}}_r}+\mybox{\delta^{\overline{\sigma},y}_r}\hat{V}_r\right)\right)- \\
-\hat{Z}_r^\top\left(\delta^{\overline{\sigma},\overline{x}}_r+\delta^{\overline{\sigma},y}_r\hat{V}_r+\delta^{\overline{\sigma},z}_r\left(I_{d}-\hat{V}_r\delta^{\overline{\sigma},z}_r\right)^{-1}\left(\hat{Z}_r+
\hat{V}_r\left(\delta^{\overline{\sigma},\overline{x}}_r+\delta^{\overline{\sigma},y}_r\hat{V}_r\right)\right)\right)\bigg)\dx r,
\end{multline*}
where the marked terms effectively disappear for the following reason:

Remember that $\delta^{\overline{\sigma},z}_r$, interpreted as a $d\times d$ - matrix, consists of an upper left $d_1\times d_1$ block and a lower right $d_2\times d_2$ block. The remaining components of the matrix always vanish. This implies that the matrix $\left(I_{d}-\hat{V}_r\delta^{\overline{\sigma},z}_r\right)^{-1}$ will also have this structure. \\
Remember also that the first $d_1$ components of the vector $\delta^{f,z}$ vanish, which means that $$\delta^{f,z}_r\left(I_{d}-\hat{V}_r\delta^{\overline{\sigma},z}_r\right)^{-1}$$ will also have this property.\\
Finally the last $d_2$ components of $\delta^{\overline\sigma,x}$ and $\delta^{\overline\sigma,y}$ vanish. Therefore, the expressions
$$\delta^{f,z}_r\left(I_{d}-\hat{V}_r\delta^{\overline{\sigma},z}_r\right)^{-1}\cdot\mybox{\delta^{\overline{\sigma},\overline{x}}_r}\quad\textrm{ and }\quad
\delta^{f,z}_r\left(I_{d}-\hat{V}_r\delta^{\overline{\sigma},z}_r\right)^{-1}\cdot\mybox{\delta^{\overline{\sigma},y}_r}$$ vanish. $\hfill\myqed$

We finally obtain
\begin{multline}\label{vhatOld}
\hat{V}_{s\wedge\tau}=\hat{V}_{\tau}-\int_{s\wedge\tau}^{\tau}\dx W_r^\top\hat{Z}_r-
\int_{s\wedge\tau}^{\tau}\bigg(\delta^{f,\overline{x}+y}_r+\delta^{f,\overline{x}+y}_r\hat{V}_r+\delta^{f,z}_r\left(I_{d}-\hat{V}_r\delta^{\overline{\sigma},z}_r\right)^{-1}\hat{Z}_r-\\
-\hat{Z}_r^\top\left(\delta^{\overline{\sigma},\overline{x}}_r+\delta^{\overline{\sigma},y}_r\hat{V}_r+\delta^{\overline{\sigma},z}_r\left(I_{d}-\hat{V}_r\delta^{\overline{\sigma},z}_r\right)^{-1}\left(\hat{Z}_r+
\hat{V}_r\left(\delta^{\overline{\sigma},\overline{x}}_r+\delta^{\overline{\sigma},y}_r\hat{V}_r\right)\right)\right)\bigg)\dx r
\end{multline}
a.s. for all $s\in[t_1,T]$, for almost all $x\in\mathbb{R}$. Note that this BSDE is quadratic in $\hat{Z}$. \\
Now, let us demonstrate that we can actually set $\tau=T$: Let
$$ \tau_0:=\inf\{s\in[t_1,T]\,|\,U_s\leq 0\}\wedge T. $$
Note here that $U$ is a continuous process starting in $1$.
$\hat{Z}$ is well-defined on $[t_1,\tau_0)$. Furthermore, due to (\ref{vhatOld}) and the boundedness of $\hat{V}$ on $[t_1,T]$, $\hat{Z}$ is a BMO-process on $[t_1,\tau]$ for every stopping time $\tau<\tau_0$, with a $BMO(\mathbb{P})$ - norm which can be controlled independently of $\tau$ (Theorem A.1.11. in \cite{Fromm2015}). Because of (\ref{BS2}) the process $\gamma_r:=\delta^{\overline{\sigma},\overline{x}}_r+\delta^{\overline{\sigma},y}_r\hat{V}_r+\delta^{\overline{\sigma},z}_r\tilde{Z}_r\hat{U}_r$ is also a BMO-process with the same property.
This implies $\mathbb{E}\left[\int_{t_1}^{\tau_0}|\gamma_r|^2\dx r\right]<\infty$. Remember
$$\check{U}_{s\wedge\tau}=1+\int_{t_1}^{s\wedge\tau}\dx W_r^\top\gamma_r \check{U}_r,\quad s\in [t_1,T] $$
for all stopping times $\tau\in[t_1,\tau_0)$. This implies
$$ \check{U}_{\tau}=\exp\left(\int_{t_1}^{\tau}\dx W_r^\top\gamma_r-\frac{1}{2}\int_{t_1}^{\tau}|\gamma_r|^2\dx r\right) $$
for all stopping times $\tau\in[t_1,\tau_0]$. Because of continuity of $\check{U}$ this implies
$$ \check{U}_{\tau_0}=\exp\left(\int_{t_1}^{\tau_0}\dx W_r^\top\gamma_r-\frac{1}{2}\int_{t_1}^{\tau_0}|\gamma_r|^2\dx r\right)>0. $$
Now, note that $\{\tau_0<T\}\subseteq\{\check{U}_{\tau_0}=0\}$. However, since $\check{U}_{\tau_0}>0$ a.s., as we have seen, $\tau_0=T$ a.s. must hold. Therefore, $\check{U}_{T}>0$ and even $\check{U}_{s}>0$ for all $s\in[t_1,T]$, so we can indeed set $\tau=T$ in (\ref{vhatOld}). $\hfill\myqed$
\vspace*{1mm}

Due to \eqref{vhatOld} and the positivity of $\check{U}$ we have

\begin{multline}\label{vhat}
\hat{V}_{s}=\hat{V}_{T}-\int_{s}^{T}\dx W^\top_r\hat{Z}_r-
\int_{s}^{T}\bigg(\delta^{f,\overline{x}+y}_r+\delta^{f,\overline{x}+y}_r\hat{V}_r+\delta^{f,z}_r\left(I_{d}-\hat{V}_r\delta^{\overline{\sigma},z}_r\right)^{-1}\hat{Z}_r- \\
-\hat{Z}_r^\top\left(\delta^{\overline{\sigma},\overline{x}}_r+\hat{V}_r\delta^{\overline{\sigma},y}_r+\delta^{\overline{\sigma},z}_r\left(I_{d}-\hat{V}_r\delta^{\overline{\sigma},z}_r\right)^{-1}\left(\hat{Z}_r+
\hat{V}_r\left(\delta^{\overline{\sigma},\overline{x}}_r+\delta^{\overline{\sigma},y}_r\hat{V}_r\right)\right)\right)\bigg)\dx r,
\end{multline}


Due to the hypotheses of the theorem, $\hat{V}_r$ and $\left(I_{d}-\hat{V}_r\delta^{\overline{\sigma},z}_r\right)^{-1}$ are uniformly bounded by constants, which do not depend on $t_1$! \\
Furthermore, according to the It\^o formula and the dynamics of $\frac{\dx}{\dx \tilde{x}}\tilde{X}$ given by (\ref{tildetildex}) its (matrix) inverse $\tilde{U}_r:=\left(\frac{\dx}{\dx \tilde{x}}\tilde{X}_r\right)^{-1}$ exists and has dynamics
\begin{equation}\label{invdyn}
\tilde{U}_s=I_N+\int_{t_1}^s\tilde{U}_r\left(\sum_{i=1}^d\delta^{\tilde{\sigma},\tilde{x},i}_r\delta^{\tilde{\sigma},\tilde{x},i}_r-
\delta^{\tilde{\mu},\tilde{x}}_r\right)\dx r-\sum_{i=1}^d\int_{t_1}^s\dx W^i_r\tilde{U}_r\delta^{\tilde{\sigma},\tilde{x},i}_r,
\end{equation}
a.s. for all $s\in[t_1,T]$. In particular, we can assume that $\tilde{U}$ is continuous in time.

Using the chain rule of Lemma A.3.1. in \cite{Fromm2015} we have
\begin{multline*} \frac{\dx}{\dx \tilde{x}}Y_s=\frac{\dx}{\dx \tilde{x}}\left(u(s,\tilde{X}_s,\overline{X}_s)\right)=\frac{\dx}{\dx \tilde{x}}u(s,\tilde{X}_s,\overline{X}_s)\frac{\dx}{\dx \tilde{x}}\tilde{X}_s+
\frac{\dx}{\dx \overline{x}}u(s,\tilde{X}_s,\overline{X}_s)\frac{\dx}{\dx \tilde{x}}\overline{X}_s=\\
=\frac{\dx}{\dx \tilde{x}}u(s,\tilde{X}_s,\overline{X}_s)\frac{\dx}{\dx \tilde{x}}\tilde{X}_s+\hat{V}_s\frac{\dx}{\dx \tilde{x}}\overline{X}_s.
\end{multline*}
Now, define
\begin{itemize}
\item $R_s:=\frac{\dx}{\dx \tilde{x}}Y_s-\hat{V}_s\frac{\dx}{\dx \tilde{x}}\overline{X}_s=\frac{\dx}{\dx \tilde{x}}u(s,\tilde{X}_s,\overline{X}_s)\frac{\dx}{\dx \tilde{x}}\tilde{X}_s$ and
\item $\tilde{R}_s:=R_s\tilde{U}_s=\frac{\dx}{\dx \tilde{x}}u(s,\tilde{X}_s,\overline{X}_s)$,
\end{itemize}
which are both $1\times (1\times N)$ - valued.
Using the It\^o formula we can deduce the dynamics of $R$ and then of $\tilde{R}$. Let us first deal with $R$. We rely on (\ref{tildey}), (\ref{tildeoverx}) and (\ref{vhat}):
$$ R_s=R_T-\int_{s}^{T}\dx W^\top_r\check{Z}_r-\int_{s}^{T}H_r\dx r, $$
where
$$ \check{Z}_r:=\frac{\dx}{\dx \tilde{x}}Z_r-\hat{Z}_r\frac{\dx}{\dx \tilde{x}}\overline{X}_r-\hat{V}_r\left(\varepsilon\delta^{\overline{\sigma},\tilde{x}}_r\frac{\dx}{\dx \tilde{x}}\tilde{X}_r+\delta^{\overline{\sigma},\overline{x}}_r\frac{\dx}{\dx \tilde{x}}\overline{X}_r+
\delta^{\overline{\sigma},y}_r\frac{\dx}{\dx \tilde{x}}Y_r+\delta^{\overline{\sigma},z}_r\frac{\dx}{\dx \tilde{x}}Z_r\right) $$
and where $H_r$ is another auxiliary process defined as
\begin{multline*}
 H_r:=\varepsilon\delta^{f,\tilde{x}}_r\frac{\dx}{\dx \tilde{x}}\tilde{X}_r+\delta^{f,\overline{x}+y}_r\left(\mybox{\frac{\dx}{\dx \tilde{x}}\overline{X}_r}+\frac{\dx}{\dx \tilde{x}}Y_r\right)+\delta^{f,z}_r\frac{\dx}{\dx \tilde{x}}Z_r- \\
-\Bigg\{\mybox{\delta^{f,\overline{x}+y}_r}+\delta^{f,\overline{x}+y}_r\hat{V}_r+\delta^{f,z}_r\left(I_{d}-\hat{V}_r\delta^{\overline{\sigma},z}_r\right)^{-1}\hat{Z}_r- \\
-\hat{Z}_r^\top\left(\Mybox{\delta^{\overline{\sigma},\overline{x}}_r}+\hat{V}_r\delta^{\overline{\sigma},y}_r+\delta^{\overline{\sigma},z}_r\left(I_{d}-\hat{V}_r\delta^{\overline{\sigma},z}_r\right)^{-1}\left(\hat{Z}_r+
\hat{V}_r\left(\delta^{\overline{\sigma},\overline{x}}_r+\delta^{\overline{\sigma},y}_r\hat{V}_r\right)\right)\right)\Bigg\}\frac{\dx}{\dx \tilde{x}}\overline{X}_r-\\
-\hat{V}_r\varepsilon\delta^{\overline{\mu},\tilde{x}}_r\frac{\dx}{\dx \tilde{x}}\tilde{X}_r-
\hat{Z}^\top_r\left(\varepsilon\delta^{\overline{\sigma},\tilde{x}}_r\frac{\dx}{\dx \tilde{x}}\tilde{X}_r+\Mybox{\delta^{\overline{\sigma},\overline{x}}_r\frac{\dx}{\dx \tilde{x}}\overline{X}_r}+
\delta^{\overline{\sigma},y}_r\frac{\dx}{\dx \tilde{x}}Y_r+\delta^{\overline{\sigma},z}_r\frac{\dx}{\dx \tilde{x}}Z_r\right).
\end{multline*}
In the above expression the marked terms effectively cancel out and we obtain:
\begin{multline*}
 H_r=\varepsilon\delta^{f,\tilde{x}}_r\frac{\dx}{\dx \tilde{x}}\tilde{X}_r+\mybox{\delta^{f,\overline{x}+y}_r\frac{\dx}{\dx \tilde{x}}Y_r}+\delta^{f,z}_r\frac{\dx}{\dx \tilde{x}}Z_r-
\Bigg\{\mybox{\delta^{f,\overline{x}+y}_r\hat{V}_r}+\delta^{f,z}_r\left(I_{d}-\hat{V}_r\delta^{\overline{\sigma},z}_r\right)^{-1}\hat{Z}_r- \\
-\hat{Z}_r^\top\left(\Mybox{\hat{V}_r\delta^{\overline{\sigma},y}_r}+\delta^{\overline{\sigma},z}_r\left(I_{d}-\hat{V}_r\delta^{\overline{\sigma},z}_r\right)^{-1}\left(\hat{Z}_r+
\hat{V}_r\left(\delta^{\overline{\sigma},\overline{x}}_r+\delta^{\overline{\sigma},y}_r\hat{V}_r\right)\right)\right)\Bigg\}\frac{\dx}{\dx \tilde{x}}\overline{X}_r-\\
-\hat{V}_r\varepsilon\delta^{\overline{\mu},\tilde{x}}_r\frac{\dx}{\dx \tilde{x}}\tilde{X}_r-\hat{Z}^\top_r\left(\varepsilon\delta^{\overline{\sigma},\tilde{x}}_r\frac{\dx}{\dx \tilde{x}}\tilde{X}_r+
\Mybox{\delta^{\overline{\sigma},y}_r\frac{\dx}{\dx \tilde{x}}Y_r}+\delta^{\overline{\sigma},z}_r\frac{\dx}{\dx \tilde{x}}Z_r\right).
\end{multline*}
In the above expression the marked terms can be effectively merged using $\frac{\dx}{\dx \tilde{x}}Y_r-\hat{V}_s\frac{\dx}{\dx \tilde{x}}\overline{X}_r=R_r$, so we can further simplify:
\begin{multline}\label{hwozhat}
 H_r=\varepsilon\delta^{f,\tilde{x}}_r\frac{\dx}{\dx \tilde{x}}\tilde{X}_r+\delta^{f,\overline{x}+y}_rR_r+\delta^{f,z}_r\mybox{\frac{\dx}{\dx \tilde{x}}Z_r}-\delta^{f,z}_r\left(I_{d}-\hat{V}_r\delta^{\overline{\sigma},z}_r\right)^{-1}\hat{Z}_r\frac{\dx}{\dx \tilde{x}}\overline{X}_r+\\
+\hat{Z}_r^\top\bigg(\delta^{\overline{\sigma},z}_r\left(I_{d}-\hat{V}_r\delta^{\overline{\sigma},z}_r\right)^{-1}\left(\hat{Z}_r+
\hat{V}_r\left(\delta^{\overline{\sigma},\overline{x}}_r+\delta^{\overline{\sigma},y}_r\hat{V}_r\right)\right)\bigg)\frac{\dx}{\dx \tilde{x}}\overline{X}_r-\\
-\hat{V}_r\varepsilon\delta^{\overline{\mu},\tilde{x}}_r\frac{\dx}{\dx \tilde{x}}\tilde{X}_r-\hat{Z}^\top_r\left(\varepsilon\delta^{\overline{\sigma},\tilde{x}}_r\frac{\dx}{\dx \tilde{x}}\tilde{X}_r+
\delta^{\overline{\sigma},y}_rR_r+\delta^{\overline{\sigma},z}_r\mybox{\frac{\dx}{\dx \tilde{x}}Z_r}\right).
\end{multline}
Now, using the definition of $\check{Z}$ we can write
$$ \check{Z}_r+\hat{Z}_r\frac{\dx}{\dx \tilde{x}}\overline{X}_r+\hat{V}_r\left(\varepsilon\delta^{\overline{\sigma},\tilde{x}}_r\frac{\dx}{\dx \tilde{x}}\tilde{X}_r+\delta^{\overline{\sigma},\overline{x}}_r\frac{\dx}{\dx \tilde{x}}\overline{X}_r+
\delta^{\overline{\sigma},y}_r\frac{\dx}{\dx \tilde{x}}Y_r\right)=\frac{\dx}{\dx \tilde{x}}Z_r-\hat{V}_r\delta^{\overline{\sigma},z}_r\frac{\dx}{\dx \tilde{x}}Z_r, $$
$$ \frac{\dx}{\dx \tilde{x}}Z_r=\left(I_{d}-\hat{V}_r\delta^{\overline{\sigma},z}_r\right)^{-1}\left(\check{Z}_r+\hat{Z}_r\frac{\dx}{\dx \tilde{x}}\overline{X}_r+\hat{V}_r\left(\varepsilon\delta^{\overline{\sigma},\tilde{x}}_r\frac{\dx}{\dx \tilde{x}}\tilde{X}_r+\delta^{\overline{\sigma},\overline{x}}_r\frac{\dx}{\dx \tilde{x}}\overline{X}_r+
\delta^{\overline{\sigma},y}_r\frac{\dx}{\dx \tilde{x}}Y_r\right)\right)= $$
$$=\left(I_{d}-\hat{V}_r\delta^{\overline{\sigma},z}_r\right)^{-1}\left(\check{Z}_r+\hat{V}_r\varepsilon\delta^{\overline{\sigma},\tilde{x}}_r\frac{\dx}{\dx \tilde{x}}\tilde{X}_r+
\hat{Z}_r\frac{\dx}{\dx \tilde{x}}\overline{X}_r+
\hat{V}_r\left(\delta^{\overline{\sigma},\overline{x}}_r\frac{\dx}{\dx \tilde{x}}\overline{X}_r+
\delta^{\overline{\sigma},y}_r\mybox{\frac{\dx}{\dx \tilde{x}}Y_r}\right)\right).$$
We again use $\frac{\dx}{\dx \tilde{x}}Y_r=R_r+\hat{V}_s\frac{\dx}{\dx \tilde{x}}\overline{X}_r$:
\begin{multline}\label{zcheckplug}
\frac{\dx}{\dx \tilde{x}}Z_r=\left(I_{d}-\hat{V}_r\delta^{\overline{\sigma},z}_r\right)^{-1}\cdot \\
\shoveright{\cdot\left(\check{Z}_r+\hat{V}_r\varepsilon\delta^{\overline{\sigma},\tilde{x}}_r\frac{\dx}{\dx \tilde{x}}\tilde{X}_r+
\hat{Z}_r\frac{\dx}{\dx \tilde{x}}\overline{X}_r+
\hat{V}_r\left(\delta^{\overline{\sigma},\overline{x}}_r\frac{\dx}{\dx \tilde{x}}\overline{X}_r+
\delta^{\overline{\sigma},y}_rR_r+\delta^{\overline{\sigma},y}_r\hat{V}_s\frac{\dx}{\dx \tilde{x}}\overline{X}_r\right)\right)=} \\
\shoveleft{=\left(I_{d}-\hat{V}_r\delta^{\overline{\sigma},z}_r\right)^{-1}\cdot} \\
\shoveright{\cdot\left(\check{Z}_r+\hat{V}_r\varepsilon\delta^{\overline{\sigma},\tilde{x}}_r\frac{\dx}{\dx \tilde{x}}\tilde{X}_r+\hat{V}_r\delta^{\overline{\sigma},y}_rR_r+
\hat{Z}_r\mybox{\frac{\dx}{\dx \tilde{x}}\overline{X}_r}+
\hat{V}_r\left(\delta^{\overline{\sigma},\overline{x}}_r\mybox{\frac{\dx}{\dx \tilde{x}}\overline{X}_r}+\delta^{\overline{\sigma},y}_r\hat{V}_s\mybox{\frac{\dx}{\dx \tilde{x}}\overline{X}_r}\right)\right)=} \\
=\left(I_{d}-\hat{V}_r\delta^{\overline{\sigma},z}_r\right)^{-1}\left(\check{Z}_r+\hat{V}_r\varepsilon\delta^{\overline{\sigma},\tilde{x}}_r\frac{\dx}{\dx \tilde{x}}\tilde{X}_r+\hat{V}_r\mybox{\delta^{\overline{\sigma},y}_r}R_r+
\left(\hat{Z}_r+\hat{V}_r\left(\mybox{\delta^{\overline{\sigma},\overline{x}}_r}+\mybox{\delta^{\overline{\sigma},y}_r}\hat{V}_s\right)\right)\frac{\dx}{\dx \tilde{x}}\overline{X}_r\right),
\end{multline}
where in the last step we used the distributive law. \\
Now, let us plug (\ref{zcheckplug}) into (\ref{hwozhat}) at the first place where $\frac{\dx}{\dx \tilde{x}}Z_r$ appears: Remember that the first $d_1$ components of $\delta^{f,z}$ and the last $d_2$ components of $\delta^{\overline{\sigma},\overline{x}}$ and $\delta^{\overline{\sigma},y}$ vanish. So, the expressions
$$ \delta^{f,z}_r\left(I_{d}-\hat{V}_r\delta^{\overline{\sigma},z}_r\right)^{-1}\cdot\mybox{\delta^{\overline{\sigma},\overline{x}}_r}\quad\textrm{ and } \quad
\delta^{f,z}_r\left(I_{d}-\hat{V}_r\delta^{\overline{\sigma},z}_r\right)^{-1}\cdot\mybox{\delta^{\overline{\sigma},y}_r} $$
vanish and we end up with:
\begin{multline*}
 H_r=\varepsilon\delta^{f,\tilde{x}}_r\frac{\dx}{\dx \tilde{x}}\tilde{X}_r+\delta^{f,\overline{x}+y}_rR_r+\delta^{f,z}_r\left(I_{d}-\hat{V}_r\delta^{\overline{\sigma},z}_r\right)^{-1}\left(\check{Z}_r+\hat{V}_r\varepsilon\delta^{\overline{\sigma},\tilde{x}}_r\frac{\dx}{\dx \tilde{x}}\tilde{X}_r+\mybox{\hat{Z}_r\frac{\dx}{\dx \tilde{x}}\overline{X}_r}\right)-\\
-\delta^{f,z}_r\left(I_{d}-\hat{V}_r\delta^{\overline{\sigma},z}_r\right)^{-1}\mybox{\hat{Z}_r\frac{\dx}{\dx \tilde{x}}\overline{X}_r}+\hat{Z}_r^\top\bigg(\delta^{\overline{\sigma},z}_r\left(I_{d}-\hat{V}_r\delta^{\overline{\sigma},z}_r\right)^{-1}\left(\hat{Z}_r+
\hat{V}_r\left(\delta^{\overline{\sigma},\overline{x}}_r+\delta^{\overline{\sigma},y}_r\hat{V}_r\right)\right)\bigg)\frac{\dx}{\dx \tilde{x}}\overline{X}_r-\\
-\hat{V}_r\varepsilon\delta^{\overline{\mu},\tilde{x}}_r\frac{\dx}{\dx \tilde{x}}\tilde{X}_r-\hat{Z}^\top_r\left(\varepsilon\delta^{\overline{\sigma},\tilde{x}}_r\frac{\dx}{\dx \tilde{x}}\tilde{X}_r+
\delta^{\overline{\sigma},y}_rR_r+\delta^{\overline{\sigma},z}_r\frac{\dx}{\dx \tilde{x}}Z_r\right).
\end{multline*}
Note that the marked terms above effectively cancel out, so we obtain:
\begin{multline}\label{hrwosecz}
 H_r=\varepsilon\delta^{f,\tilde{x}}_r\frac{\dx}{\dx \tilde{x}}\tilde{X}_r+\delta^{f,\overline{x}+y}_rR_r+\delta^{f,z}_r\left(I_{d}-\hat{V}_r\delta^{\overline{\sigma},z}_r\right)^{-1}\left(\check{Z}_r+\hat{V}_r\varepsilon\delta^{\overline{\sigma},\tilde{x}}_r\frac{\dx}{\dx \tilde{x}}\tilde{X}_r\right)+\\
+\mybox{\hat{Z}_r^{\top}\delta^{\overline{\sigma},z}_r\left(I_{d}-\hat{V}_r\delta^{\overline{\sigma},z}_r\right)^{-1}\left(\hat{Z}_r+
\hat{V}_r\left(\delta^{\overline{\sigma},\overline{x}}_r+\delta^{\overline{\sigma},y}_r\hat{V}_r\right)\right)\frac{\dx}{\dx \tilde{x}}\overline{X}_r}-\\
-\hat{V}_r\varepsilon\delta^{\overline{\mu},\tilde{x}}_r\frac{\dx}{\dx \tilde{x}}\tilde{X}_r-\Mybox{\hat{Z}^\top_r}\left(\varepsilon\delta^{\overline{\sigma},\tilde{x}}_r\frac{\dx}{\dx \tilde{x}}\tilde{X}_r+
\delta^{\overline{\sigma},y}_rR_r+\Mybox{\delta^{\overline{\sigma},z}_r\frac{\dx}{\dx \tilde{x}}Z_r}\right).
\end{multline}
But according to (\ref{zcheckplug}) we can write
\begin{multline*}
\Mybox{\hat{Z}_r^{\top}\delta^{\overline{\sigma},z}_r\frac{\dx}{\dx \tilde{x}}Z_r}=\hat{Z}_r^{\top}\delta^{\overline{\sigma},z}_r\left(I_{d}-\hat{V}_r\delta^{\overline{\sigma},z}_r\right)^{-1}\left(\check{Z}_r+\hat{V}_r\varepsilon\delta^{\overline{\sigma},\tilde{x}}_r\frac{\dx}{\dx \tilde{x}}\tilde{X}_r+\hat{V}_r\delta^{\overline{\sigma},y}_rR_r\right)+ \\
+\mybox{\hat{Z}_r^{\top}\delta^{\overline{\sigma},z}_r\left(I_{d}-\hat{V}_r\delta^{\overline{\sigma},z}_r\right)^{-1}\left(\hat{Z}_r+\hat{V}_r\left(\delta^{\overline{\sigma},\overline{x}}_r+\delta^{\overline{\sigma},y}_r\hat{V}_s\right)\right)\frac{\dx}{\dx \tilde{x}}\overline{X}_r},
\end{multline*}
so plugging this into (\ref{hrwosecz}) leads to:
\begin{multline*}
 H_r=\varepsilon\delta^{f,\tilde{x}}_r\frac{\dx}{\dx \tilde{x}}\tilde{X}_r+\delta^{f,\overline{x}+y}_rR_r+\delta^{f,z}_r\left(I_{d}-\hat{V}_r\delta^{\overline{\sigma},z}_r\right)^{-1}\left(\check{Z}_r+\hat{V}_r\varepsilon\delta^{\overline{\sigma},\tilde{x}}_r\frac{\dx}{\dx \tilde{x}}\tilde{X}_r\right)- \\
-\hat{V}_r\varepsilon\delta^{\overline{\mu},\tilde{x}}_r\frac{\dx}{\dx \tilde{x}}\tilde{X}_r-\hat{Z}^\top_r\left(\mybox{\varepsilon\delta^{\overline{\sigma},\tilde{x}}_r\frac{\dx}{\dx \tilde{x}}\tilde{X}_r}+
\mybox{\delta^{\overline{\sigma},y}_rR_r}\right)- \\
-\hat{Z}^\top_r\delta^{\overline{\sigma},z}_r\left(I_{d}-\hat{V}_r\delta^{\overline{\sigma},z}_r\right)^{-1}\left(\check{Z}_r+\hat{V}_r\mybox{\varepsilon\delta^{\overline{\sigma},\tilde{x}}_r\frac{\dx}{\dx \tilde{x}}\tilde{X}_r}+\hat{V}_r\mybox{\delta^{\overline{\sigma},y}_rR_r}\right).
\end{multline*}
Note that we have $I_{d}+\hat{V}_r\delta^{\overline{\sigma},z}_r\left(I_{d}-\hat{V}_r\delta^{\overline{\sigma},z}_r\right)^{-1}=\left(I_{d}-\hat{V}_r\delta^{\overline{\sigma},z}_r\right)^{-1}$, which can be straightforwardly verified by multiplying both sides of this equation with $I_{d}-\hat{V}_r\delta^{\overline{\sigma},z}_r$ from the right. Using this relationship we can combine the marked terms above, obtaining:
\begin{multline*}
 H_r=\varepsilon\delta^{f,\tilde{x}}_r\mybox{\frac{\dx}{\dx \tilde{x}}\tilde{X}_r}+\delta^{f,\overline{x}+y}_r\Mybox{R_r}+\delta^{f,z}_r\left(I_{d}-\hat{V}_r\delta^{\overline{\sigma},z}_r\right)^{-1}\left(\check{Z}_r+\hat{V}_r\varepsilon\delta^{\overline{\sigma},\tilde{x}}_r\mybox{\frac{\dx}{\dx \tilde{x}}\tilde{X}_r}\right)- \\
-\hat{V}_r\varepsilon\delta^{\overline{\mu},\tilde{x}}_r\mybox{\frac{\dx}{\dx \tilde{x}}\tilde{X}_r}-\hat{Z}^\top_r\left(I_{d}-\hat{V}_r\delta^{\overline{\sigma},z}_r\right)^{-1}\left(\varepsilon\delta^{\overline{\sigma},\tilde{x}}_r\mybox{\frac{\dx}{\dx \tilde{x}}\tilde{X}_r}+
\delta^{\overline{\sigma},y}_r\Mybox{R_r}\right)- \\
-\hat{Z}^\top_r\delta^{\overline{\sigma},z}_r\left(I_{d}-\hat{V}_r\delta^{\overline{\sigma},z}_r\right)^{-1}\check{Z}_r.
\end{multline*}
Using the distributive law we obtain
\begin{multline*}
 H_r=
\left(\varepsilon\delta^{f,\tilde{x}}_r+\delta^{f,z}_r\left(I_{d}-\hat{V}_r\delta^{\overline{\sigma},z}_r\right)^{-1}\hat{V}_r\varepsilon\mybox{\delta^{\overline{\sigma},\tilde{x}}_r}-\hat{V}_r\varepsilon\delta^{\overline{\mu},\tilde{x}}_r-\hat{Z}^\top_r\left(I_{d}-\hat{V}_r\delta^{\overline{\sigma},z}_r\right)^{-1}\varepsilon\mybox{\delta^{\overline{\sigma},\tilde{x}}_r}\right)\frac{\dx}{\dx \tilde{x}}\tilde{X}_r+ \\
+\left(\delta^{f,\overline{x}+y}_r-\hat{Z}^\top_r\left(I_{d}-\hat{V}_r\delta^{\overline{\sigma},z}_r\right)^{-1}\delta^{\overline{\sigma},y}_r\right)R_r+ \\
+\delta^{f,z}_r\Mybox{\left(I_{d}-\hat{V}_r\delta^{\overline{\sigma},z}_r\right)^{-1}\check{Z}_r}-\hat{Z}^\top_r\delta^{\overline{\sigma},z}_r\Mybox{\left(I_{d}-\hat{V}_r\delta^{\overline{\sigma},z}_r\right)^{-1}\check{Z}_r},
\end{multline*}
which further simplifies to
\begin{multline*}
 H_r=
\varepsilon\left(\delta^{f,\tilde{x}}_r+\left(\hat{V}_r\delta^{f,z}_r-\hat{Z}^\top_r\right)\left(I_{d}-\hat{V}_r\delta^{\overline{\sigma},z}_r\right)^{-1}\delta^{\overline{\sigma},\tilde{x}}_r-\hat{V}_r\delta^{\overline{\mu},\tilde{x}}_r\right)\frac{\dx}{\dx \tilde{x}}\tilde{X}_r+ \\
+\left(\delta^{f,\overline{x}+y}_r-\hat{Z}^\top_r\left(I_{d}-\hat{V}_r\delta^{\overline{\sigma},z}_r\right)^{-1}\delta^{\overline{\sigma},y}_r\right)R_r+ \\
+\left(\delta^{f,z}_r-\hat{Z}^\top_r\delta^{\overline{\sigma},z}_r\right)\left(I_{d}-\hat{V}_r\delta^{\overline{\sigma},z}_r\right)^{-1}\check{Z}_r.
\end{multline*}
Let us now deduce the dynamics of $\tilde{R}=R\tilde{U}$. We use the dynamics of $R$ we just obtained, as well as (\ref{invdyn}), which describes dynamics of $\tilde{U}$:
$$ \tilde{R}_s=\tilde{R}_T-\int_{s}^{T}\dx W^\top_r\dot{Z}_r-\int_{s}^{T}G_r\dx r, $$
where $\dot{Z}$ is $\mathbb{R}^{d\times (1\times N)}$ - valued and can be written as $\left(\dot{Z}^i\right)_{i=1,\ldots,d}$ such that
$$ \dot{Z}^i_r=\check{Z}^i_r\tilde{U}_r-R_r\tilde{U}_r\delta^{\tilde{\sigma},\tilde{x},i}_r=\mybox{\check{Z}^i_r\tilde{U}_r}-\tilde{R}_r\delta^{\tilde{\sigma},\tilde{x},i}_r, $$
for $i=1,\ldots,d$, and
$$ G_r=H_r\tilde{U}_r+R_r\tilde{U}_r\left(\sum_{i=1}^d\delta^{\tilde{\sigma},\tilde{x},i}_r\delta^{\tilde{\sigma},\tilde{x},i}_r-
\delta^{\tilde{\mu},\tilde{x}}_r\right)-\sum_{i=1}^d\mybox{\check{Z}^i_r\tilde{U}_r}\delta^{\tilde{\sigma},\tilde{x},i}_r= $$
$$ =H_r\tilde{U}_r+\tilde{R}_r\left(\sum_{i=1}^d\delta^{\tilde{\sigma},\tilde{x},i}_r\delta^{\tilde{\sigma},\tilde{x},i}_r-
\delta^{\tilde{\mu},\tilde{x}}_r\right)-\sum_{i=1}^d\left(\dot{Z}^i_r+\tilde{R}_r\delta^{\tilde{\sigma},\tilde{x},i}_r\right)\delta^{\tilde{\sigma},\tilde{x},i}_r
=\mybox{H_r}\tilde{U}_r-\tilde{R}_r\delta^{\tilde{\mu},\tilde{x}}_r-\sum_{i=1}^d\dot{Z}^i_r\delta^{\tilde{\sigma},\tilde{x},i}_r. $$
By plugging in the structure of $H$ and using $\left(\frac{\dx}{\dx\tilde{x}}\tilde{X}_r\right)\tilde{U}_r=I_N$ we obtain
\begin{multline*}
G_r=\varepsilon\left(\delta^{f,\tilde{x}}_r+\left(\hat{V}_r\delta^{f,z}_r-\hat{Z}^\top_r\right)\left(I_{d}-\hat{V}_r\delta^{\overline{\sigma},z}_r\right)^{-1}\delta^{\overline{\sigma},\tilde{x}}_r-\hat{V}_r\delta^{\overline{\mu},\tilde{x}}_r\right)+\\
+\left(\delta^{f,\overline{x}+y}_r-\hat{Z}^\top_r\left(I_{d}-\hat{V}_r\delta^{\overline{\sigma},z}_r\right)^{-1}\delta^{\overline{\sigma},y}_r\right)\mybox{\tilde{R}_r}+ \\
+\left(\delta^{f,z}_r-\hat{Z}^\top_r\delta^{\overline{\sigma},z}_r\right)\left(I_{d}-\hat{V}_r\delta^{\overline{\sigma},z}_r\right)^{-1}\Mybox{\check{Z}_r\tilde{U}_r}
-\mybox{\tilde{R}_r}\delta^{\tilde{\mu},\tilde{x}}_r-\sum_{i=1}^d\dot{Z}^i_r\delta^{\tilde{\sigma},\tilde{x},i}_r.
\end{multline*}
Knowing $\Mybox{\check{Z}^i_r\tilde{U}_r}=\dot{Z}^i_r+\tilde{R}_r\delta^{\tilde{\sigma},\tilde{x},i}_r$ we obtain
\begin{multline*}
G_r=\varepsilon\left(\delta^{f,\tilde{x}}_r+\left(\hat{V}_r\delta^{f,z}_r-\hat{Z}^\top_r\right)\left(I_{d}-\hat{V}_r\delta^{\overline{\sigma},z}_r\right)^{-1}\delta^{\overline{\sigma},\tilde{x}}_r-\hat{V}_r\delta^{\overline{\mu},\tilde{x}}_r\right)+\\
+\mybox{\tilde{R}_r}\left(-\delta^{\tilde{\mu},\tilde{x}}_r+\left(\delta^{f,\overline{x}+y}_r-\hat{Z}^\top_r\left(I_{d}-\hat{V}_r\delta^{\overline{\sigma},z}_r\right)^{-1}\delta^{\overline{\sigma},y}_r\right)I_N\right)+ \\
+\left(\delta^{f,z}_r-\hat{Z}^\top_r\delta^{\overline{\sigma},z}_r\right)\left(I_{d}-\hat{V}_r\delta^{\overline{\sigma},z}_r\right)^{-1}\dot{Z}_r+ \\
+\sum_{i=1}^d\left(\left(\delta^{f,z}_r-\hat{Z}^\top_r\delta^{\overline{\sigma},z}_r\right)\left(I_{d}-\hat{V}_r\delta^{\overline{\sigma},z}_r\right)^{-1}\right)^i\mybox{\tilde{R}_r}\delta^{\tilde{\sigma},\tilde{x},i}_r
-\sum_{i=1}^d\dot{Z}^i_r\delta^{\tilde{\sigma},\tilde{x},i}_r,
\end{multline*}
where $z^{i}$ refers to the $i$ - th component of a vector $z\in\mathbb{R}^{1\times d}$. \\
We can rewrite using distributive law:
\begin{multline}\label{Gexp}
G_r=\varepsilon\left(\delta^{f,\tilde{x}}_r+\left(\hat{V}_r\delta^{f,z}_r-\hat{Z}^\top_r\right)\left(I_{d}-\hat{V}_r\delta^{\overline{\sigma},z}_r\right)^{-1}\delta^{\overline{\sigma},\tilde{x}}_r-\hat{V}_r\delta^{\overline{\mu},\tilde{x}}_r\right)+\\
\shoveleft{+\tilde{R}_r\Bigg\{-\delta^{\tilde{\mu},\tilde{x}}_r+\left(\delta^{f,\overline{x}+y}_r-\hat{Z}^\top_r\left(I_{d}-\hat{V}_r\delta^{\overline{\sigma},z}_r\right)^{-1}\delta^{\overline{\sigma},y}_r\right)I_N+} \\
\shoveright{+\sum_{i=1}^d\left(\left(\delta^{f,z}_r-\hat{Z}^\top_r\delta^{\overline{\sigma},z}_r\right)\left(I_{d}-\hat{V}_r\delta^{\overline{\sigma},z}_r\right)^{-1}\right)^i\delta^{\tilde{\sigma},\tilde{x},i}_r\Bigg\}+} \\
+\left(\delta^{f,z}_r-\hat{Z}^\top_r\delta^{\overline{\sigma},z}_r\right)\left(I_{d}-\hat{V}_r\delta^{\overline{\sigma},z}_r\right)^{-1}\dot{Z}_r-\sum_{i=1}^d\dot{Z}^i_r\delta^{\tilde{\sigma},\tilde{x},i}_r.
\end{multline}
Now, remember that $Y$ satisfies
$$ Y_s=Y_T-\int_s^T\dx W^\top_r Z_r-\int_s^Tf(r,\tilde{X}_r,\overline{X}_r,Y_r,Z_r)\dx r,\qquad r\in[t_1,T]. $$
Also, $q_r:=f(r,\tilde{X}_r,\overline{X}_r,Y_r,0)$ is bounded by $\|f(\cdot,\cdot,\cdot,\cdot,0)\|_\infty$, and, furthermore, the difference
$$ f(r,\tilde{X}_r,\overline{X}_r,Y_r,Z_r)-f(r,\tilde{X}_r,\overline{X}_r,Y_r,0)=p_r Z_r, $$
where the bounded process $p$ is defined via
$$ p_r:=\left(\frac{1}{|Z_r|^2}\left(f(r,\tilde{X}_r,\overline{X}_r,Y_r,Z_r)-f(r,\tilde{X}_r,\overline{X}_r,Y_r,0)\right)Z^\top_r\right), $$
is bounded by $C(1+|Z_r|)$ due to our requirements for $\frac{\dx}{\dx z}f$.\\
The backward equation $Y_s=Y_T-\int_s^T\dx W^\top_r Z_r-\int_s^Tq_r+p_rZ_r\dx r$ together with the boundedness of $Y_T=\xi(\varepsilon\tilde{X}_T,\overline{X}_T)$ imply:
\begin{itemize}
\item $Y$ is uniformly bounded by $\|\xi\|_\infty+T\|f(\cdot,\cdot,\cdot,\cdot,0)\|_\infty$ (see Lemma A.1.10. in \cite{Fromm2015}, which is applicable since $Z$ is bounded),
\item $Z$ and therefore $\delta^{f,z}=\frac{\dx}{\dx z}f(\cdot,\tilde{X},\overline{X},Y,Z)$ are BMO - processes with $BMO(\mathbb{P})$ - norms controlled independently of $t_1$ and $\varepsilon$ (see Theorem A.1.11. in \cite{Fromm2015}). $\hfill\myqed$
\end{itemize}
Also, due to
\begin{itemize}
\item the dynamics of $\hat{V}$ given by (\ref{vhat}),
\item the uniform boundedness of $\hat{V}$ and $\left(I_{d}-\hat{V}\delta^{\overline{\sigma},z}\right)^{-1}$,
\item our requirements for $\frac{\dx}{\dx z}f$ and $\frac{\dx}{\dx (\overline{x}+y)}f$
\end{itemize}
Theorem A.1.11. in \cite{Fromm2015} is applicable to (\ref{vhat}) and we have that $\hat{Z}$ is also a BMO - process with a $BMO(\mathbb{P})$ - norm controlled independently of $t_1$ and $\varepsilon$. $\hfill\myqed$

Using (\ref{Gexp}) the process $\tilde{R}$ has dynamics
$$ \tilde{R}_s=\tilde{R}_T-\int_s^T\dx W^\top_r\dot{Z}_r-\int_s^T \left(\varepsilon\alpha_r+\tilde{R}_r\left(\delta^{f,\overline{x}+y}_rI_N+\beta_r\right)+\mu\dot{Z}_r+\sum_{i=1}^d\dot{Z}^i_r\gamma^i_r\right)\dx r, $$
where
\begin{itemize}
\item $\alpha$ is an $\mathbb{R}^{1\times (1\times N)}$-valued BMO process,
\item $\beta$ is an $\mathbb{R}^{(1\times N)\times (1\times N)}$-valued BMO process,
\item $\mu$  is an $\mathbb{R}^{1\times d}$-valued BMO process and
\item $\gamma^i$, $i=1,\ldots,d$ are bounded progressive $\mathbb{R}^{(1\times N)\times (1\times N)}$-valued processes,
\end{itemize}
such that the $BMO(\mathbb{P})$ - norms of $\alpha,\beta,\mu$ and the supremum norms of $\gamma^i$ can be controlled independently of $t_1$ and $\varepsilon$. \\
Also, note the relationship $\tilde{R}_T=\varepsilon\frac{\dx}{\dx\tilde{x}}\xi(\varepsilon\tilde{X}_T,\overline{X}_T)$, which is a direct consequence of the terminal condition $u(T,\tilde{x},\overline{x})=\xi(\varepsilon\tilde{x},\overline{x})$. So, $\tilde{R}_T$ is bounded by $\varepsilon L_{\xi,\tilde{x}}$. \\
We know that $\tilde{R}_s=\frac{\dx}{\dx \tilde{x}}u(s,\tilde{X}_s,\overline{X}_s)$ is a bounded process but not necessarily bounded independently of $t_1$, $\varepsilon$ (at this point). However, we can now apply Lemma A.1.7. in \cite{Fromm2015} to obtain
$$ \|\tilde{R}\|_\infty\leq C\varepsilon L_{\xi,\tilde{x}}+C\varepsilon\|\alpha\|_{BMO(\mathbb{P})}, $$
where $C\in (0,\infty)$ depends only on $T$, $\|\mu\|_{BMO(\mathbb{P})}$, $\|\beta\|_{BMO(\mathbb{P})}$ and $\|\gamma\|_\infty$ and is monotonically increasing in these values.

This shows that for $\varepsilon>0$ small enough $L_{u(t_1,\cdot),(\tilde{x},\overline{x})^\top}\leq K+\varepsilon \tilde{C}< L_{\overline{\sigma},z}^{-1}$ will hold independently of $t_1$, where $\tilde{C}$ is a constant, which does not depend on $t_1$ and $\varepsilon$. This contradicts the statement of Lemma \ref{EXPlosionM}. Therefore, the assumption $I^{M}_{\mathrm{max}}=(t^{M}_{\mathrm{min}},T]$ was wrong and so, $I^{M}_{\mathrm{max}}=[0,T]$ for $\varepsilon>0$ small enough is proven.
\end{proof}

\subsection{Main result}\label{mainresult}

Now, let us apply the above abstract result to solve the actual FBSDE (\ref{utility-FBSDE}): 

We want to investigate the solvability of the forward backward system given by the forward equation
$$ \tilde{X}_t=\tilde{x}+\int_{0}^t\tilde{\mu}(s,\tilde{X}_s)\dx s+\int_{0}^t\dx W^\top_s\tilde{\sigma}(s,\tilde{X}_s), $$
$$ X_t=x-\int_0^t\left(\dx W_s+\pi_{1}(\tilde\theta(s,\tilde{X}_s))\dx s\right)^\top\left(\pi_{1}(\tilde\theta(s,\tilde{X}_s))\frac{U'}{U''}(X_s+Y_s)+\pi_{1}(Z_s)\right)$$
and the backward equation
\begin{multline*}
Y_t=\tilde{H}(\tilde{X}_T,X_T)-\int_t^T\left(\dx W_s+\pi_{1}(\tilde\theta(s,\tilde{X}_s))\dx s\right)^\top Z_s-\\
-\int_t^T\left(|\pi_{1}(\tilde\theta(s,\tilde{X}_s))|^2\frac{U'}{U''}\left(1-\frac{1}{2}\frac{U^{(3)}U'}{(U'')^2}\right)(X_s+Y_s)-
\frac{1}{2}|\pi_{2}(Z_s)|^2\cdot\frac{U^{(3)}}{U''}(X_s+Y_s)\right)\dx s,
\end{multline*}
$t\in [0,T]$, where $\tilde{x}\in\mathbb{R}^{1\times N}$, $N\in\mathbb{N}$, $x\in\mathbb{R}$.

So, the problem is about finding progressively measurable processes $\tilde{X},X,Y,Z$ such that $\tilde{X}$ is $\mathbb{R}^{1\times N}$ - valued, $X$ and $Y$ are both $\mathbb{R}$ - valued while $Z$ is $\mathbb{R}^{d}$ - valued and such that the above FBSDE is satisfied.

We assume that $U$ satisfies condition (C1) and $\tilde{\mu},\tilde{\sigma},\tilde{H},\tilde{\theta}$ satisfy condition (C2).

\begin{thm}\label{finalresult}
Under these conditions the above problem has a unique solution $\tilde{X},X,Y,Z$ on $[0,T]$ satisfying $\|Z\|_\infty<\infty$.\\
Furthermore, the problem can be reduced to an MLLC problem with $I^{M}_{\mathrm{max}}=[0,T]$.
\end{thm}
\begin{proof}
Note that the forward equation for $\tilde{X}$ has a unique solution which can be obtained independently of the other parts of the problem, since $\tilde{X}$ satisfies a rather standard SDE. So, our task is really about establishing existence and uniqueness of $X,Y,Z$. \\
We define a Brownian motion with drift $B$ via
$$ B_s=W_s+\int_0^s \pi_{1}(\tilde\theta(r,\tilde{X}_r))\dx r,\quad s\in[0,T]. $$
Note that $B$ is a Brownian motion under some probability measure $\mathbb{Q}\sim\mathbb{P}$. $\tilde{X}$ has dynamics
$$ \tilde{X}_t=\tilde{x}+\int_{0}^t\left(\tilde{\mu}-\pi_1(\tilde\theta)^\top\tilde{\sigma}\right)(r,\tilde{X}_r)\dx r+\int_{0}^t\dx B^\top_r\tilde{\sigma}(r,\tilde{X}_r), $$
which describes a uniquely solvable Lipschitz problem, so $\tilde{X}$ is adapted w.r.t. the filtration generated by $B$ (and augmented by $\mathcal{F}_0$), which in turn implies that $W=B-\int_0^\cdot \pi_{1}(\tilde\theta(r,\tilde{X}_r)\dx r$ is adapted w.r.t. the filtration generated by $B$ as well. So, $W$ and $B$ generate the same filtration $(\mathcal{F}_t)_{t\in[0,T]}$. \\
We now introduce a slightly modified problem: For $\varepsilon>0$ consider the system given by the forward equation
$$ \check{X}_s=\check{x}+\int_{t}^s\frac{1}{\varepsilon}\left(\tilde{\mu}-\pi_1(\tilde\theta)^\top\tilde{\sigma}\right)(r,\varepsilon\check{X}_r)\dx r+\int_{t}^s\dx B^\top_r\frac{1}{\varepsilon}\tilde{\sigma}(r,\varepsilon\check{X}_r), $$
\begin{equation}\label{forward2} X_s=x-\int_t^s\dx B_r^\top\left(\pi_{1}(\tilde\theta(r,\varepsilon\check{X}_r))\frac{U'}{U''}(X_r+Y_r)+\pi_{1}(Z_r)\right), \end{equation}
and the backward equation
\begin{multline}\label{backwardH}
Y_s=\tilde{H}(\varepsilon\check{X}_T,X_T)-\int_s^T\dx B_r^\top Z_r-\\
-\int_s^T\left(|\pi_{1}(\tilde\theta(r,\varepsilon\check{X}_r))|^2\frac{U'}{U''}\left(1-\frac{1}{2}\frac{U^{(3)}U'}{(U'')^2}\right)(X_r+Y_r)-
\frac{1}{2}|\pi_{2}(Z_r)|^2\cdot\frac{U^{(3)}}{U''}(X_r+Y_r)\right)\dx r,
\end{multline}
$s\in[t,T]$.
This new forward-backward system is completely equivalent to the preceding one in the sense that $\tilde{X},X,Y,Z$ solve the initial system if and only if $\check{X}:=\frac{1}{\varepsilon}\tilde{X},X,Y,Z$ solve the new system. So, it remains to show, that for some $\varepsilon>0$ the new system will have a unique solution with bounded $Z$. For that purpose we apply Theorem \ref{sc3} to show that for the above problem $I^{M}_{\mathrm{max}}=[0,T]$ will hold for some $\varepsilon>0$. Let us first check, that the system satisfies the structural requirements of Theorem \ref{sc3}: \\
Firstly, observe the following properties of $U$:
\begin{itemize}
\item According to Lemma \ref{Uproperties} the functions $\frac{U'}{U''},\frac{U''}{U'}$ and $\frac{U^{(3)}}{U''}$ are bounded,
\item $(\ln(-U''))''=-\kappa''$ is non-positive and bounded,
\item $U''$ is point-wise negative. Furthermore,
\item $\frac{U'}{U''}$, $\frac{U'}{U''}\left(1-\frac{1}{2}\frac{U^{(3)}U'}{(U'')^2}\right)$ and $\frac{U^{(3)}}{U''}$ are bounded and Lipschitz continuous:

Note that the product or the sum of two bounded and Lipschitz continuous functions is again bounded and Lipschitz continuous. Thus we only need boundedness and Lipschitz continuity of $\frac{U'}{U''}$ and $\frac{U^{(3)}}{U''}$. Boundedness is already known. Furthermore,
\begin{equation}\label{upeq1}  \left(\frac{U'}{U''}\right)'=\frac{(U'')^2-U'U^{(3)}}{(U'')^2}=\mybox{1-\frac{U'}{U''}\cdot\frac{U^{(3)}}{U''}} \,\textrm{ and also} \end{equation}
$$ (\ln(-U''))'=\frac{-U^{(3)}}{-U''}=\frac{U^{(3)}}{U''},\textrm{ which implies}$$
\begin{equation}\label{upeq2}  \left(\frac{U^{(3)}}{U''}\right)'=\mybox{(\ln(-U''))''}, \end{equation}
where both marked expressions are bounded. $\hfill\myqed$
\item $\left(\frac{U'}{U''}\left(1-\frac{1}{2}\frac{U^{(3)}U'}{(U'')^2}\right)\right)'\geq 0$ and $\left(\frac{U^{(3)}}{U''}\right)'\leq 0$:

The second inequality is clear, since $\left(\frac{U^{(3)}}{U''}\right)'=(\ln(-U''))''\leq 0$ as we saw. The first requires a bit more calculation: Using the product rule together with (\ref{upeq1}) and (\ref{upeq2}):
$$\left(\frac{U'}{U''}\left(1-\frac{1}{2}\frac{U^{(3)}U'}{(U'')^2}\right)\right)'=\left(1-\frac{U'}{U''}\cdot\frac{U^{(3)}}{U''}\right)\left(1-\frac{1}{2}\frac{U^{(3)}U'}{(U'')^2}\right)-$$
$$-\frac{U'}{U''}\frac{1}{2}\left((\ln(-U''))''\frac{U'}{U''}+\frac{U^{(3)}}{U''}\left(1-\frac{U'}{U''}\cdot\frac{U^{(3)}}{U''}\right)\right)=$$
$$=\left(1-\frac{U'}{U''}\cdot\frac{U^{(3)}}{U''}\right)\left(1-\frac{1}{2}\frac{U^{(3)}U'}{(U'')^2}-\frac{U'}{U''}\frac{1}{2}\frac{U^{(3)}}{U''}\right)-\frac{U'}{U''}\frac{1}{2}(\ln(-U''))''\frac{U'}{U''}= $$
$$=\left(1-\frac{U'}{U''}\cdot\frac{U^{(3)}}{U''}\right)^2-\frac{1}{2}\left(\frac{U'}{U''}\right)^2(\ln(-U''))''\geq 0. $$
$\hfill\myqed$
\end{itemize}
Using the notation of the previous section the parameter functions $\overline{\mu},\overline{\sigma},f$ implied by the above problem (\ref{forward2}), (\ref{backwardH}) satisfy:
\begin{itemize}
\item $\overline{\mu}$ vanishes,
\item $\overline\sigma$ and $f$ are differentiable in $\check{x},x,y,z$ such that all of the partial derivatives are uniformly bounded except for $\frac{\dx}{\dx (x+y)}f$ and $\frac{\dx}{\dx z}f$. This boundedness comes from boundedness and Lipschitz continuity of $\tilde\theta$ together with the aforementioned properties of $U$.
\end{itemize}
The generator $f$ of the backward equation satisfies the structural requirements of Theorem \ref{sc3}:
\begin{itemize}
\item It is a function of $s,\check{x},x+y,\pi_2(z)$ and has quadratic growth in $\pi_2(z)$, while
\item its derivative w.r.t. $z$ has linear growth in $\pi_2(z)$ and
\item its derivative w.r.t. $x+y$ has quadratic growth in $\pi_2(z)$. It is also non-negative according to the aforementioned properties of $U$. Finally,
\item $f(\cdot,\cdot,\cdot,\cdot,0)$ is uniformly bounded.
\end{itemize}
The parameter functions $\overline\sigma$ and $\tilde{H}$ also have the structure required by Theorem \ref{sc3}: Note here that $\pi_2(\overline\sigma)=0$ and $\overline{\sigma}$ is a function of $s,\check{x},x+y,\pi_1(z)$. Also, $L_{\overline\sigma,z}=1$ such that $L_{\tilde{H},x}<1=L_{\sigma,z}^{-1}$.
\vspace{2mm}

In order to apply Theorem \ref{sc3} we merely need to control $\frac{\dx}{\dx x}u$ uniformly for every weakly regular Markovian decoupling field $u:[t,T]\times\mathbb{R}^{1\times N}\times\mathbb{R}\rightarrow\mathbb{R}$ to the above problem for small $\varepsilon>0$. This control has to be independent of $u$, $t$ and $\varepsilon$. \\
For this purpose we seek to control $\hat{V}_r:=\frac{\dx}{\dx x}u(r,\check{X}_r,X_r)$, $r\in[t,T]$. Note that w.l.o.g. $\hat{V}$ is bounded by $L_{u,(\tilde{x},x)}<L_{\overline\sigma,z}^{-1}=1$.
Using notations from the proof of Theorem \ref{sc3} we have similarly to (\ref{vhat}):
\begin{multline}\label{vhat2}
\hat{V}_{s}=\hat{V}_{T}-\int_{s}^{T}\dx B^\top_r\hat{Z}_r-
\int_{s}^{T}\Bigg\{\delta^{f,x+y}_r+\delta^{f,x+y}_r\hat{V}_r+\mybox{\delta^{f,z}_r\left(I_{d}-\hat{V}_r\delta^{\overline{\sigma},z}_r\right)^{-1}}\hat{Z}_r- \\
-\hat{Z}_r^\top\left(\delta^{\overline\sigma,x}_r+\hat{V}_r\delta^{\overline{\sigma},y}_r+\mybox{\delta^{\overline{\sigma},z}_r\left(I_{d}-\hat{V}_r\delta^{\overline{\sigma},z}_r\right)^{-1}}\left(\hat{Z}_r+
\hat{V}_r\left(\delta^{\overline\sigma,x}_r+\delta^{\overline{\sigma},y}_r\hat{V}_r\right)\right)\right)\Bigg\}\dx r,\quad s\in[t,T].\,
\end{multline}
In our case $\delta^{\overline{\sigma},z}_r$ is equal to the diagonal $d\times d$ - matrix having the value $-1$ in the first $d_1$ diagonal entries and $0$ everywhere else. Therefore,
\begin{itemize}
\item $\left(I_{d}-\hat{V}_r\delta^{\overline{\sigma},z}_r\right)^{-1}$ is a diagonal matrix having $(1+\hat{V}_r)^{-1}=\frac{1}{1+\hat{V}_r}$ in the first $d_1$ diagonal entries and $1$ in the others. So,
\item $\delta^{f,z}_r\left(I_{d}-\hat{V}_r\delta^{\overline{\sigma},z}_r\right)^{-1}=\delta^{f,z}_r$, since the first $d_1$ diagonal entries of $\delta^{f,z}_r$ vanish. Furthermore,
\item $\delta^{\overline\sigma,z}_r\left(I_{d}-\hat{V}_r\delta^{\overline{\sigma},z}_r\right)^{-1}$ is a diagonal matrix having $-(1+\hat{V}_r)^{-1}$ in the first $d_1$ diagonal entries and $0$ everywhere else.
\end{itemize}
So, we can simplify (\ref{vhat2}):
\begin{multline*}
\hat{V}_{s}=\hat{V}_{T}-\int_{s}^{T}\dx B^\top_r\hat{Z}_r-
\int_{s}^{T}\bigg(\delta^{f,x+y}_r+\delta^{f,x+y}_r\hat{V}_r+\delta^{f,z}_r\hat{Z}_r- \\
-\hat{Z}_r^\top\left(\mybox{\delta^{\overline\sigma,x}_r+\hat{V}_r\delta^{\overline{\sigma},y}_r}-\left(1+\hat{V}_r\right)^{-1}\pi_1\left(\hat{Z}_r+
\hat{V}_r\left(\mybox{\delta^{\overline\sigma,x}_r+\delta^{\overline{\sigma},y}_r\hat{V}_r}\right)\right)\right)\bigg)\dx r.
\end{multline*}
Now, use $\delta^{\overline{\sigma},y}_r=\delta^{\overline\sigma,x}_r=\pi_1(\delta^{\overline\sigma,x}_r)$ to see
\begin{itemize}
\item $\delta^{\overline\sigma,x}_r+\delta^{\overline{\sigma},y}_r\hat{V}_r=\delta^{\overline\sigma,x}_r(1+\hat{V}_r)$ and
\item $\delta^{\overline\sigma,x}_r+\hat{V}_r\delta^{\overline{\sigma},y}_r-\left(1+\hat{V}_r\right)^{-1}\pi_1\left(\hat{V}_r\left(\delta^{\overline\sigma,x}_r+\delta^{\overline{\sigma},y}_r\hat{V}_r\right)\right)=\delta^{\overline\sigma,x}_r+\hat{V}_r\delta^{\overline{\sigma},x}_r-\pi_1\left(\hat{V}_r\delta^{\overline{\sigma},x}_r\right)=\delta^{\overline\sigma,x}_r$,
\end{itemize}
such that we obtain
\begin{multline*}
\hat{V}_{s}=\hat{V}_{T}-\int_{s}^{T}\dx B^\top_r\hat{Z}_r-\int_{s}^{T}\bigg(\delta^{f,x+y}_r+\delta^{f,x+y}_r\hat{V}_r+\delta^{f,z}_r\hat{Z}_r-\\
-\hat{Z}_r^\top\left(\delta^{\overline\sigma,x}_r\right)+\hat{Z}_r^\top\left(1+\hat{V}_r\right)^{-1}\pi_1\left(\hat{Z}_r\right)\bigg)\dx r,
\end{multline*}
or
$$
\hat{V}_{s}=\hat{V}_{T}-\int_{s}^{T}\dx B^\top_r\hat{Z}_r-
\int_{s}^{T}\bigg(\delta^{f,x+y}_r(1+\hat{V}_r)+\left(\delta^{f,z}_r-\left(\delta^{\overline\sigma,x}_r\right)^\top\right)\hat{Z}_r+\left(1+\hat{V}_r\right)^{-1}\left|\pi_1\left(\hat{Z}_r\right)\right|^2\bigg)\dx r.
$$
Now, apply the It\^o formula to $\ln(1+\hat{V}_s)$:
\begin{multline*}
\ln(1+\hat{V}_s)=\ln(1+\hat{V}_T)-\int_{s}^{T}\dx B^\top_r(1+\hat{V}_r)^{-1}\hat{Z}_r- \\
-\int_{s}^{T}\bigg(\left(1+\hat{V}_r\right)^{-1}\left(\delta^{f,x+y}_r(1+\hat{V}_r)+\left(\delta^{f,z}_r-\left(\delta^{\overline\sigma,x}_r\right)^\top\right)\hat{Z}_r+\left(1+\hat{V}_r\right)^{-1}\left|\pi_1\left(\hat{Z}_r\right)\right|^2\right)- \\
-\frac{1}{2}\left(1+\hat{V}_r\right)^{-2}|\hat{Z}_r|^2\bigg)\dx r
\end{multline*}
which after defining $\overline{Z}:=(1+\hat{V})^{-1}\hat{Z}$ simplifies to
\begin{multline}\label{lndym}
\ln(1+\hat{V}_s)=\ln(1+\hat{V}_T)-\int_{s}^{T}\dx B^\top_r\overline{Z}_r- \\
-\int_{s}^{T}\bigg(\delta^{f,x+y}_r+\left(\delta^{f,z}_r-\left(\delta^{\overline\sigma,x}_r\right)^\top\right)\overline{Z}_r+
\left|\pi_1\left(\overline{Z}_r\right)\right|^2-\frac{1}{2}|\overline{Z}_r|^2\bigg)\dx r.
\end{multline}
Let us rewrite this equation as
\begin{multline*}
\ln(1+\hat{V}_s)=\ln(1+\hat{V}_T)-\int_{s}^{T}\left(\dx B_r+\left(\left(\delta^{f,z}_r\right)^\top-\delta^{\overline\sigma,x}_r+\pi_1(\overline{Z}_r)-\frac{1}{2}\overline{Z}_r\right)\dx r\right)^\top\overline{Z}_r- \\
-\int_{s}^{T}\delta^{f,x+y}_r\dx r.
\end{multline*}
Since $\ln(1+\hat{V})$ is a bounded process, $\overline{Z}$ is a BMO process under $\mathbb{Q}$ according to \eqref{lndym} and Theorem A.1.11. in \cite{Fromm2015}. Furthermore, $Z$ and, thereby, $\delta^{f,z}$ is bounded. Therefore, using a Girsanov measure change we get after exploiting $\delta^{f,x+y}\geq 0$:
\begin{multline*}
\ln\left(1+\frac{\dx}{\dx x}u(t,\check{x},x)\right)=\mathbb{E}_{\mathbb{Q}_1}\left[\ln\left(1+\hat{V}_s\right)\right]\leq \\
\leq \mathbb{E}_{\mathbb{Q}_1}\left[\ln\left(1+\hat{V}_T\right)\right]\leq \ln\left(1+\left\|\frac{\dx}{\dx x}H\right\|_\infty\right)=\ln(1+L_{H,x}),
\end{multline*}
under some probability measure $\mathbb{Q}_1\sim\mathbb{Q}\sim\mathbb{P}$. This simplifies to $\frac{\dx}{\dx x}u(t,\check{x},x)\leq L_{\tilde{H},x}<L_{\sigma,z}^{-1}$ for almost all $\check{x},x$. Similarly $\frac{\dx}{\dx x}u(s,\cdot,\cdot)\leq L_{\tilde{H},x}<L_{\sigma,z}^{-1}$ a.e. for $s\in[t,T]$, since the same arguments can be applied to the weakly regular Markovian decoupling field $u|_{[s,T]}$.
\vspace{2mm}

Uniformly controlling $\frac{\dx}{\dx x}u$ from below is, however, a bit more challenging and will be based on a rather deep exploitation of the specific structure of the forward-backward system:

Define a Brownian motion with drift via
$$ \tilde{B}_s:=B_s-B_t+\int_t^s\left(\left(\delta^{f,z}_r\right)^\top-\delta^{\overline\sigma,x}_r+\pi_1\left(\overline{Z}_r\right)\right)\dx r $$
The BSDE (\ref{lndym}) can also be rewritten as
$$
\ln(1+\hat{V}_s)=\ln(1+\hat{V}_T)-\int_{s}^{T}\dx \tilde{B}^\top_r\overline{Z}_r-\int_{s}^{T}\bigg(\delta^{f,x+y}_r-\frac{1}{2}\left|\overline{Z}_r\right|^2\bigg)\dx r.
$$
$\tilde{B}$ is a Brownian motion under some probability measure $\tilde{\mathbb{Q}}\sim\mathbb{Q}$ (see Theorem 2.3.\ in \cite{kazam}), so
$$
\mathbb{E}_{\tilde{\mathbb{Q}}}[\ln(1+\hat{V}_s)]=\mathbb{E}_{\tilde{\mathbb{Q}}}[\ln(1+\hat{V}_T)]-\mathbb{E}_{\tilde{\mathbb{Q}}}\left[\int_{s}^{T}\left(\delta^{f,x+y}_r-\frac{1}{2}\left|\overline{Z}_r\right|^2\right)\dx r\right],\quad s\in[t,T].
$$
In order to control $\mathbb{E}_{\tilde{\mathbb{Q}}}[\ln(1+\hat{V}_t)]$ from below we need to control $\mathbb{E}_{\tilde{\mathbb{Q}}}\left[\int_{t}^{T}\left(\delta^{f,x+y}_r-\frac{1}{2}\left|\overline{Z}_r\right|^2\right)\dx r\right]$ from above.
Remembering the structure of $f$ we have:
\begin{equation}\label{fxy}\delta^{f,x+y}_r=|\pi_{1}(\tilde\theta(r,\varepsilon\check{X}_r))|^2\left(\frac{U'}{U''}\left(1-\frac{1}{2}\frac{U^{(3)}U'}{(U'')^2}\right)\right)'(X_r+Y_r)-
\frac{1}{2}|\pi_{2}(Z_r)|^2\cdot\left(\frac{U^{(3)}}{U''}\right)'(X_r+Y_r). \end{equation}
Now, define $P:=X+Y$. By summing up the forward equation for $X$ and the backward equation (\ref{backwardH}) we obtain the dynamics of $P$:
\begin{multline*}
P_s=P_T-\int_s^T\dx B_r^\top \left(\pi_{2}(Z_r)-\pi_{1}(\tilde\theta(r,\varepsilon\check{X}_r))\frac{U'}{U''}(P_r)\right)-\\
-\int_s^T\left(|\pi_{1}(\tilde\theta(r,\varepsilon\check{X}_r))|^2\frac{U'}{U''}\left(1-\frac{1}{2}\frac{U^{(3)}U'}{(U'')^2}\right)(P_r)-
\frac{1}{2}|\pi_{2}(Z_r)|^2\cdot\frac{U^{(3)}}{U''}(P_r)\right)\dx r,
\end{multline*}
Now, define $\varphi:=\frac{U'}{U''}$. Clearly, $\varphi$ is negative. We also know that it is bounded and also bounded away from $0$. Remember $\kappa=-\ln(-U'')$, so $\kappa'=-\frac{U^{(3)}}{U''}$. According to the It\^o formula the bounded process $\varphi(P)$ has dynamics
\begin{multline*}
\varphi(P_s)=\varphi(P_T)-\int_s^T\dx B_r^\top \varphi'(P_r)\left(\pi_{2}(Z_r)-\pi_{1}(\tilde\theta(r,\varepsilon\check{X}_r))\varphi(P_r)\right)-\\
-\int_s^T\Bigg\{\varphi'(P_r)\left(|\pi_{1}(\tilde\theta(r,\varepsilon\check{X}_r))|^2\varphi\cdot\left(1+\frac{1}{2}\varphi\kappa'\right)(P_r)+ \frac{1}{2}|\pi_{2}(Z_r)|^2\kappa'(P_r)\right)+ \\
+\frac{1}{2}\varphi''(P_r)\mybox{\left|\pi_{2}(Z_r)-\pi_{1}(\tilde\theta(r,\varepsilon\check{X}_r))\varphi(P_r)\right|^2}\Bigg\}\dx r.
\end{multline*}
Note
$$\left|\pi_{2}(Z_r)-\pi_{1}(\tilde\theta(r,\varepsilon\check{X}_r))\varphi(P_r)\right|^2=\left|\pi_{2}(Z_r)\right|^2+\left|\pi_{1}(\tilde\theta(r,\varepsilon\check{X}_r))\varphi(P_r)\right|^2, $$
due to orthogonality. So, after regrouping the terms we have
\begin{multline}\label{philast}
\varphi(P_r)=\varphi(P_T)-\int_s^T\dx B_r^\top \varphi'(P_r)\left(\pi_{2}(Z_r)-\pi_{1}(\tilde\theta(r,\varepsilon\check{X}_r))\varphi(P_r)\right)-\\
-\int_s^T\Bigg(|\pi_{1}(\tilde\theta(r,\varepsilon\check{X}_r))|^2\left(\varphi'\varphi\cdot\left(1+\frac{1}{2}\varphi\kappa'\right)+\frac{1}{2}\varphi''\varphi^2\right)(P_r)+ \frac{1}{2}|\pi_{2}(Z_r)|^2\left(\varphi'\kappa'+\varphi''\right)(P_r)\Bigg)\dx r.
\end{multline}
Now, consider the definition of $\tilde{B}$. Due to the structure of $f$ and $\overline\sigma$ we have $\delta^{f,z}_r=\pi_2(Z_r)^\top\kappa'(P_r)$ and $\delta^{\overline\sigma,x}_r=-\pi_{1}(\tilde\theta(r,\varepsilon\check{X}_r))\varphi'(P_r)$, so
\begin{multline*}
\int_s^T\dx\tilde{B}_r^\top \varphi'(P_r)\left(\pi_{2}(Z_r)-\pi_{1}(\tilde\theta(r,\varepsilon\check{X}_r))\varphi(P_r)\right)=\\
=\int_s^T\dx B_r^\top \varphi'(P_r)\left(\pi_{2}(Z_r)-\pi_{1}(\tilde\theta(r,\varepsilon\check{X}_r))\varphi(P_r)\right)+ \\
+\int_s^T\left(|\pi_2(Z_r)|^2\varphi'\kappa'(P_r)-|\pi_{1}(\tilde\theta(r,\varepsilon\check{X}_r))|^2\varphi(\varphi')^2(P_r)-\pi_1(\overline{Z}_r)^{\top}\pi_{1}(\tilde\theta(r,\varepsilon\check{X}_r))\varphi'\varphi(P_r)\right)\dx r,
\end{multline*}
or
\begin{multline*}
-\int_s^T\dx B_r^\top \varphi'(P_r)\left(\pi_{2}(Z_r)-\pi_{1}(\tilde\theta(r,\varepsilon\check{X}_r))\varphi(P_r)\right)=\\
=-\int_s^T\dx\tilde{B}_r^\top \varphi'(P_r)\left(\pi_{2}(Z_r)-\pi_{1}(\tilde\theta(r,\varepsilon\check{X}_r))\varphi(P_r)\right)- \\
-\int_s^T\left(-|\pi_2(Z_r)|^2\varphi'\kappa'(P_r)+|\pi_{1}(\tilde\theta(r,\varepsilon\check{X}_r))|^2\varphi(\varphi')^2(P_r)+\pi_1(\overline{Z}_r)^{\top}\pi_{1}(\tilde\theta(r,\varepsilon\check{X}_r))\varphi'\varphi(P_r)\right)\dx r,
\end{multline*}
which together with (\ref{philast}) yields
\begin{multline*}
\varphi(P_s)=\varphi(P_T)-\int_s^T\dx\tilde{B}_r^\top \varphi'(P_r)\left(\pi_{2}(Z_r)-\pi_{1}(\tilde\theta(r,\varepsilon\check{X}_r))\varphi(P_r)\right)-\\
-\int_s^T\Bigg(|\pi_{1}(\tilde\theta(r,\varepsilon\check{X}_r))|^2\left(\varphi'\varphi\left(1+\varphi'+\frac{1}{2}\varphi\kappa'\right)+\frac{1}{2}\varphi''\varphi^2\right)(P_r)+  \\
+\frac{1}{2}|\pi_{2}(Z_r)|^2\left(\mybox{\varphi'\kappa'+\varphi''-2\kappa'\varphi'}\right)(P_r)+\varphi'\varphi(P_r)\pi_{1}(\tilde\theta(r,\varepsilon\check{X}_r))^\top\overline{Z}_r\Bigg)\dx r.
\end{multline*}
We have using the chain rule:
\begin{itemize}
\item $\varphi'=\frac{U''U''-U'U^{(3)}}{(U'')^2}=1+\varphi\kappa'$, which is bounded. Furthermore,
\item $\varphi''=\varphi'\kappa'+\varphi\kappa''$, which is also bounded. Finally
\item $\varphi'\kappa'+\varphi''-2\kappa'\varphi'=\varphi'\kappa'+\varphi'\kappa'+\varphi\kappa''-2\kappa'\varphi'=\varphi\kappa''$.
\end{itemize}
And so we have after applying conditional expectations
$$ \mathbb{E}_{\tilde{\mathbb{Q}}}\left[\varphi(P_t)\right]=\mathbb{E}_{\tilde{\mathbb{Q}}}\left[\varphi(P_T)\right]-
\mathbb{E}_{\tilde{\mathbb{Q}}}\left[\int_t^T\Bigg(\alpha_s+ \frac{1}{2}\varphi\kappa''(P_s)|\pi_{2}(Z_s)|^2+\beta^\top_s\overline{Z}_s\Bigg)\dx s\right], $$
with some uniformly bounded progressively measurable processes $\alpha$, $\beta$. This means that
$$ 0\leq\mathbb{E}_{\tilde{\mathbb{Q}}}\left[\int_t^T\frac{-1}{2}\varphi\kappa''(P_s)|\pi_{2}(Z_s)|^2\dx s\right]\leq
\mathbb{E}_{\tilde{\mathbb{Q}}}\left[\int_t^T\beta^\top_s\overline{Z}_s\dx s\right]+T\|\alpha\|_\infty+2\|\varphi\|_\infty. $$
Now, note that $\delta^{f,x+y}_r=\frac{1}{2}|\pi_{2}(Z_r)|^2\kappa''(P_r)+\gamma_r$, with some uniformly bounded process $\gamma$, according to (\ref{fxy}).
This implies considering $\varphi<0$:
$$ \mathbb{E}_{\tilde{\mathbb{Q}}}\left[\int_{t}^{T}\left(\delta^{f,x+y}_r-\frac{1}{2}\left|\overline{Z}_r\right|^2\right)\dx r\right]=
\mathbb{E}_{\tilde{\mathbb{Q}}}\left[\int_{t}^{T}\left(\frac{1}{-\varphi}\frac{-1}{2}\varphi\kappa''(P_r)|\pi_{2}(Z_r)|^2+\gamma_r-\frac{1}{2}\left|\overline{Z}_r\right|^2\right)\dx r\right]\leq $$
$$ \leq\left\|\frac{1}{-\varphi}\right\|_\infty\left(\mybox{\mathbb{E}_{\tilde{\mathbb{Q}}}\left[\int_t^T\beta^\top_s\overline{Z}_s\dx s\right]}+T\|\alpha\|_\infty+2\|\varphi\|_\infty\right)+T\|\gamma\|_\infty-
\frac{1}{2}\mathbb{E}_{\tilde{\mathbb{Q}}}\left[\int_{t}^{T}\left|\overline{Z}_r\right|^2\dx r\right]\leq $$
$$ \leq\mybox{\mathbb{E}_{\tilde{\mathbb{Q}}}\left[\int_t^T\frac{1}{2}\left|\left\|\frac{1}{-\varphi}\right\|_\infty\beta_s\right|^2\dx s\right]+\mathbb{E}_{\tilde{\mathbb{Q}}}\left[\int_t^T\frac{1}{2}|\overline{Z}_s|^2\dx s\right]}+$$
$$+\left\|\frac{1}{-\varphi}\right\|_\infty\left(T\|\alpha\|_\infty+2\|\varphi\|_\infty\right)+T\|\gamma\|_\infty-
\frac{1}{2}\mathbb{E}_{\tilde{\mathbb{Q}}}\left[\int_{t}^{T}\left|\overline{Z}_r\right|^2\dx r\right]\leq $$
$$\leq \frac{1}{2}\left\|\frac{1}{-\varphi}\right\|_\infty^2T\|\beta\|_\infty^2+\left\|\frac{1}{-\varphi}\right\|_\infty\left(T\|\alpha\|_\infty+2\|\varphi\|_\infty\right)+T\|\gamma\|_\infty<\infty. $$
This is a uniform bound we were looking for. This means that
$$\mathbb{E}_{\tilde{\mathbb{Q}}}[\ln(1+\hat{V}_s)]=\ln\left(1+\frac{\dx}{\dx x}u(t,\check{x},x)\right)\geq -C$$
 for a.a. $(\check{x},x)$, where $C>0$ does not depend on $t$, or $u$ or $\varepsilon$ which immediately implies that $\frac{\dx}{\dx x}u(t,\cdot,\cdot)$ is uniformly bounded away from $-1$. The same bound works for $\frac{\dx}{\dx x}u(s,\cdot,\cdot)$, $s\in [t,T]$. $\hfill\myqed$

And so we have controlled $\frac{\dx}{\dx x}u(s,\cdot,\cdot)$ from both sides such that its modulus is bounded uniformly away from $1$ (independently of $t$, $u$, $\varepsilon$, as long as $\varepsilon$ is sufficiently small for the problem to satisfy MLLC). This shows that Theorem \ref{sc3} is applicable and we have $I^{M}_{\mathrm{max}}=[0,T]$ for some $\varepsilon>0$.
In particular, the FBSDE given by \eqref{forward2} and \eqref{backwardH} for the interval $[0,T]$ has a solution $\check{X},X,Y,Z$ s.t. $\|Z\|_\infty<\infty$ for any initial value $(\check{x},x)\in\mathbb{R}^{1\times N}\times\mathbb{R}$.\\
Furthermore, this solution is unique: Assume there is another such triple $(\check{X}',X',Y',Z')$. Then, due to boundedness of $Z'$ and the dynamics of $Y'$, the process $Y'$ must be bounded as well. At the same time the dynamics of $X'$ imply that it satisfies $\sup_{s\in [0,T]}\mathbb{E}_{\mathbb{Q}}[(X'_s)^2]<\infty$. Similar properties hold true for $X$ and $Y$, so Lemma \ref{UNIqXYZM} is applicable and the triples must coincide.
\end{proof}

\begin{remark}\label{martingalerem}
Using the It\^o formula it is straightforward to verify that the processes $X,Y,Z$ from Theorem \ref{finalresult} satisfy:
$$ U'(X_t+Y_t)=U'(X_0+Y_0)+\int_0^t U'(X_s+Y_s)\alpha^\top_s\dx W_s\qquad\textrm{a.s.}\qquad\forall t\in[0,T], $$
where $\alpha_s:=\frac{U''}{U'}(X_s+Y_s)\pi_2(Z_s)-\pi_1(\theta_s)$, $s\in[0,T]$. This implies that $t\mapsto U'(X_t+Y_t)$ describes a uniformly integrable martingale due to boundedness of $\frac{U''}{U'}$, $Z$ and $\theta$.
\end{remark}

{\small
\bibliography{quellen}{}

\begin{thebibliography}{10}

\bibitem{Antonelli1993}
F.~Antonelli.
\newblock {Backward-forward stochastic differential equations.}
\newblock {\em {Ann. Appl. Probab.}}, 3(3):777--793, 1993.

\bibitem{benderdenk}
C.~Bender and R.~Denk.
\newblock {A forward scheme for backward SDEs.}
\newblock {\em Stochastic Processes Appl.}, 117(12):1793--1812, 2007.

\bibitem{Bismut}
J.-M. Bismut.
\newblock {Th\'eorie probabiliste du contr\^ole des diffusions.}
\newblock {\em Mem. Am. Math. Soc.}, 167:130 p., 1976.

\bibitem{CvitanicKaratzas}
J.~Cvitani\'c and I.~Karatzas.
\newblock {Convex duality in constrained portfolio optimization.}
\newblock {\em Ann. Appl. Probab.}, 2(4):767--818, 1992.

\bibitem{Delarue}
F.~Delarue.
\newblock {On the existence and uniqueness of solutions to FBSDEs in a
  non-degenerate case.}
\newblock {\em Stochastic Processes Appl.}, 99(2):209--286, 2002.

\bibitem{Fromm2015}
A.~Fromm.
\newblock {\em {Theory and applications of decoupling fields for
  forward-backward stochastic differential equations}}.
\newblock PhD thesis, Humboldt-Universit\"{a}t zu Berlin, 2015.

\bibitem{Proemel2015}
A.~Fromm, P.~Imkeller, and D.~Prömel.
\newblock {An FBSDE approach to the Skorokhod embedding problem for Gaussian
  processes with non-linear drift}.
\newblock {\em Electron. J. Probab.}, 20:38 pp., 2015.

\bibitem{HIRZ13}
U.~Horst, Y.~Hu, P.~Imkeller, A.~Réveillac, and J.~Zhang.
\newblock Forward-backward systems for expected utility maximization.
\newblock {\em Stochastic Processes and their Applications}, 124(5):1813 --
  1848, 2014.

\bibitem{Hu2005}
Y.~Hu, P.~Imkeller, and M.~M{\"u}ller.
\newblock {Utility maximization in incomplete markets}.
\newblock {\em Ann. Appl. Probab.}, 15(3):1691--1712, aug 2005.

\bibitem{hupeng}
Y.~Hu and S.~Peng.
\newblock {Solution of forward-backward stochastic differential equations.}
\newblock {\em Probab. Theory Relat. Fields}, 103(2):273--283, 1995.

\bibitem{Karatzas}
I.~Karatzas, J.~P. Lehoczky, and S.~E. Shreve.
\newblock {Optimal portfolio and consumption decisions for a ``small investor''
  on a finite horizon.}
\newblock {\em SIAM J. Control Optimization}, 25:1557--1586, 1987.

\bibitem{KaratzasXu}
I.~Karatzas, J.~P. Lehoczky, S.~E. Shreve, and G.-L. Xu.
\newblock {Martingale and duality methods for utility maximization in an
  incomplete market.}
\newblock {\em SIAM J. Control Optimization}, 29(3):702--730, 1991.

\bibitem{kazam}
N.~{Kazamaki}.
\newblock {\em {Continuous exponential martingales and BMO.}}
\newblock Berlin: Springer, 1994.

\bibitem{KramkovSchachermayer}
D.~Kramkov and W.~Schachermayer.
\newblock {The asymptotic elasticity of utility functions and optimal
  investment in incomplete markets.}
\newblock {\em Ann. Appl. Probab.}, 9(3):904--950, 1999.

\bibitem{mayong}
J.~Ma, P.~Protter, and J.~Yong.
\newblock {Solving forward-backward stochastic differential equations
  explicitly -- a four step scheme.}
\newblock {\em Probab. Theory Relat. Fields}, 98(3):339--359, 1994.

\bibitem{Ma2011}
J.~Ma, Z.~Wu, D.~Zhang, and J.~Zhang.
\newblock {On Wellposedness of Forward-Backward SDEs --- A Unified Approach}.
\newblock oct 2011.

\bibitem{PT}
E.~Pardoux and S.~Tang.
\newblock {Forward-backward stochastic differential equations and quasilinear
  parabolic PDEs.}
\newblock {\em Probab. Theory Relat. Fields}, 114(2):123--150, 1999.

\bibitem{Pliska1986}
S.~R. Pliska.
\newblock {A stochastic calculus model of continuous trading: optimal
  portfolios}.
\newblock {\em Math. Oper. Res.}, 11(2):371--382, 1986.

\bibitem{ElKarouiRouge}
R.~Rouge and N.~El~Karoui.
\newblock {Pricing via utility maximization and entropy.}
\newblock {\em Math. Finance}, 10(2):259--276, 2000.

\bibitem{Sekine}
J.~Sekine.
\newblock {On exponential hedging and related quadratic backward stochastic
  differential equations.}
\newblock {\em Appl. Math. Optimization}, 54(2):131--158, 2006.

\end{thebibliography}
\bibliographystyle{abbrv}
}

\end{document}